\documentclass[11pt]{amsart}

\usepackage{macros}
\usepackage[bbgreekl]{mathbbol}
\usepackage{fullpage}
\usepackage{setspace}
\usetikzlibrary{patterns}
\numberwithin{equation}{section}

\let\emptyset\varnothing

\def\d{{\rm d}}
\def\zbar{{\overline{z}}}

\def\dt{\d t}
\def\del{\partial}
\def\delbar{{\overline{\partial}}}

\def\RRge{\RR_{\geq 0}}
\def\RRgt{\RR_{> 0}}

\def\Cur{{\rm Cur}}

\def\hotimes{{\widehat{\otimes}}}

\def\coker{{\rm coker}}
\def\bu{{\bullet}}
\def\PV{{\rm PV}}

\begin{document}
\title{Factorization algebras and abelian CS/WZW-type correspondences}

\author{Owen Gwilliam}
\address{Department of Mathematics and Statistics \\
Lederle Graduate Research Tower, 1623D \\
University of Massachusetts Amherst \\
710 N. Pleasant Street}
\email{gwilliam@math.umass.edu}

\author{Eugene Rabinovich}
\address{Department of Mathematics\\
University of California, Berkeley\\
970 Evans Hall \#3840\\
Berkeley, CA 94720}
\email{erabin@math.berkeley.edu}

\author{Brian R. Williams}
\address{School of Mathematics \\ University of Edinburgh\\ Edinburgh, UK}
\email{brian.williams@ed.ac.uk}

\begin{abstract}
We develop a method of quantization for free field theories on manifolds with boundary where the bulk theory is topological in the direction normal to the boundary and a local boundary condition is imposed. 
Our approach is within the Batalin-Vilkovisky formalism.
At the level of observables, the construction produces a stratified factorization algebra that in the bulk recovers the factorization algebra developed by Costello and Gwilliam. 
The factorization algebra on the boundary stratum enjoys a perturbative bulk-boundary correspondence with this bulk factorization algebra. 
A central example is the factorization algebra version of the abelian Chern-Simons/Wess-Zumino-Witten correspondence,
but we examine higher dimensional generalizations that are related to holomorphic truncations of string theory and $M$-theory and involve intermediate Jacobians.
\end{abstract}

\maketitle
\thispagestyle{empty}
\spacing{1}
\setcounter{tocdepth}{2}
\tableofcontents

\section{Introduction}

\spacing{1.25}

Moving from the interior to the boundary of a manifold often leads to interesting, even intricate, generalizations of constructions that make sense on manifolds without boundary. Recall, for example, the generalization of Poincar\'e duality to Lefschetz duality. In physics one likewise sees that a field theory living in the interior --- the bulk --- often couples to a theory living on the boundary to produce a rich, interacting composite system. A key example for us is the Chern-Simons/Wess-Zumino-Witten (CS/WZW) correspondence, in which a topological field theory on an oriented 3-manifold $M$ interacts with a chiral conformal field theory on its boundary $\partial M$, equipped with a complex structure to make it a Riemann surface. Here, the interaction is not via a term in the action coupling the bulk and boundary theories; instead, the interaction consists of exhibiting the boundary theory (chiral WZW theory) as a boundary condition for the bulk theory (Chern-Simons).
This correspondence originated in \cite{WittenJones}, and it has subsequently witnessed a surge of holographic generalizations in the context of the AdS/CFT correspondence. 

In this paper we revisit this kind of situation using the Batalin-Vilkovisky (BV) formalism and factorization algebras. 
It is important that we restrict to field theories that behave as topological theories in the direction normal to the boundary, as captured in Definition~\ref{dfn: tnbft}. 
Our central result is that if one imposes a local boundary condition in a homologically correct way, 
then a rather naive extension of BV quantization automatically produces a bulk-boundary correspondence, including a form of the abelian CS/WZW correspondence (cf. Theorems \ref{thm: maingenlcl} and \ref{thm: maingenlq}. We work perturbatively and see only a fragment of the full story.) The mathematical formulation is that the naive BV quantization produces a factorization algebra on the manifold whose behavior in the bulk is simply the algebra of quantum operators for the bulk theory and whose behavior on a neighborhood of the boundary is simply the algebra for the local boundary condition. 
In this introduction, we state a special case of our general theorem, with hopes it helps the reader calibrate to the discussion.

\subsection{A model case of our general result}

Our focus is on the following geometric situation. 
Let $\Sigma$ be an oriented smooth 2-dimensional manifold,
which we equip with a complex structure.
Let $M$ denote the closed half-space $\RRge \times \Sigma $,
let $\mathring M$ denote the open half-space $\RRgt \times \Sigma $,
and let $\pi: M \to \Sigma$ denote the projection map.
We view $\RRge$ as providing a kind of ``time direction'' and use $t$ to denote its coordinate.
See Figure~\ref{fig:3mfldwbdry}.

\begin{figure}
\begin{tikzpicture}
\draw[semithick](-6,0) ellipse (0.8 and 2);
\draw[semithick](-6,0.3) arc (-30:30:1);
\draw[semithick](-6,1.2) arc (150:210:0.82);
\draw[semithick](-6,-1.3) arc (-30:30:1);
\draw[semithick](-6,-0.4) arc (150:210:0.82);
\draw[semithick](-2,0) ellipse (0.8 and 2);
\draw[semithick](-2,0.3) arc (-30:30:1);
\draw[semithick](-2,1.2) arc (150:210:0.82);
\draw[semithick](-2,-1.3) arc (-30:30:1);
\draw[semithick](-2,-0.4) arc (150:210:0.82);
\draw[dotted](2,0) ellipse (0.8 and 2);
\draw[dotted](2,0.3) arc (-30:30:1);
\draw[dotted](2,1.2) arc (150:210:0.82);
\draw[dotted](2,-1.3) arc (-30:30:1);
\draw[dotted](2,-0.4) arc (150:210:0.82);
\draw[semithick](-2,2) -- (3,2);
\draw[dotted](3,2) -- (4,2);
\draw[semithick](-2,-2) -- (3,-2);
\draw[dotted](3,-2) -- (4,-2);
\node at (-7.2,1){$\Sigma$};
\node at (0,-1){$\mathring{M}$};
\node at (-3.2,1){$M$};
\draw[->,semithick] (-3,0) -- (-5,0);
\node at (-4,0.25){$\pi$};
\draw[->, semithick](-1,-2.65) -- (2,-2.65);
\node at (0.5, -2.4){$t$};

\end{tikzpicture}
\caption{Projection of $M$ onto the boundary $\Sigma$}
\label{fig:3mfldwbdry}
\end{figure}
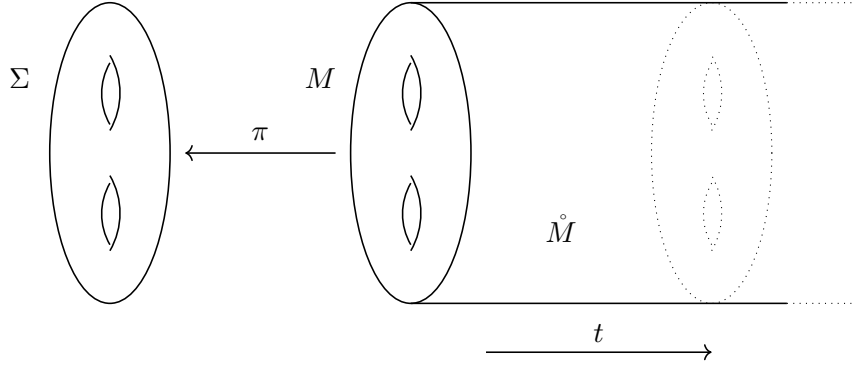

In the interior $\mathring M$, which is a manifold without boundary, 
we put (perturbative) Chern-Simons theory with gauge group~$U(1)$ with level $\kappa$.
It has a factorization algebra $\Obq_{CS}$ of quantum observables.
(See \S4.5 of~\cite{CG1}.)

On the boundary $\Sigma = \partial M$, 
there is a factorization algebra $\Cur^\q_{WZW}$ encoding the chiral $U(1)$ currents, with the Schwinger term determined by $\kappa$.
(See \S5.4 of~\cite{CG1} for its construction and verification that it recovers the standard vertex algebra and OPE.)

We construct a factorization algebra $\Obs^\q_{CS/WZW}$ for abelian Chern-Simons theory on $M$ with a particular boundary condition called the chiral WZW boundary condition. 
(As far as we are aware, 
this is the first construction of a factorization algebra of observables of a field theory arising on a manifold with boundary.)
It interpolates between the Chern-Simons observables and the chiral currents in the following precise sense.

\begin{thm}
\label{thm: main}
The factorization algebra $\Obs^\q_{CS/WZW}$ is stratified in the sense that
\begin{itemize}
\item on the interior $\mathring M$, there is a natural \emph{isomorphism} 
\[
\Obs^\q_{CS} \simeq \left(\Obs^\q_{CS/WZW}\right)\Big|_{\mathring M}
\] 
of factorization algebras, and
\item on the boundary $\partial M = \Sigma$, there is a quasi-isomorphism 
\[
\Cur^\q_{WZW} \simeq \pi_*\left(\Obs^\q_{CS/WZW}\right)
\] 
of factorization algebras. 
\end{itemize}
\end{thm}

This factorization algebra thus exhibits the desired phenomenon, 
as it is precisely the abelian Chern-Simons system in the ``bulk'' $\mathring M$ but becomes the chiral currents on the boundary~$\partial M$. 
The full factorization algebra $\Obq_{CS/WZW}$ contains more information still: it encodes an \emph{action} of the bulk observables $\Obq_{CS}$ on the boundary observables~$\Curq_{WZW}$.

There is a version of this theorem for the classical observables; 
it is a straightforward interpretation in the BV setting of the standard notion of a boundary condition for a partial differential equation.

A compelling phenomenon happens at the quantum level: 
the canonical BV quantization of abelian Chern-Simons theory in the bulk forces the appearance of the Kac-Moody cocycle $\int \alpha \wedge\partial \beta$ (i.e., Schwinger term) on the boundary.
We emphasize that these constructions are wholly rigorous, not requiring any leaps of physical intuition.
They also yield naturally a stratified factorization algebra,
and hence the theorem suggests that other bulk-boundary correspondences in the physics literature may also admit formulations in these terms.
We will describe a few such correspondences, notably a generalization of abelian CS/WZW to higher dimensions with a $4k+3$-dimensional bulk and a $4k+2$-dimensional boundary equipped with a complex structure.

One drawback of our work is that we only deal with perturbative and Lie algebraic statements here, 
not with nonperturbative and group-level versions, where many fascinating issues arise.
(As merely a jumping-off point and not a complete list of citations for this enormous subject,
we point to \cite{FMS, FFS, HopSing, KSonCS, BD, WittenJones,EMSS,FreedDQ, BecBenSchSza} as places where such issues are addressed.)
We expect that a rigorous extension of the BV formalism to global derived geometry would fold those nonperturbative issues together with our perturbative efforts.

\subsection{Consequences and applications}

One payoff here is a new view on Chern-Simons states in bundles of conformal blocks for chiral WZW models.
Factorization algebras, like sheaves, are local-to-global objects,
and so the homology of these stratified factorization algebras encode nontrivial global information.
Here, in particular, they automatically produce maps from the space of boundary observables into the global observables of the theory. 
As an example, we obtain the Chern-Simons states of the chiral WZW theory from studying the map from the boundary observables on a Riemann surface to the observables of a compact 3-manifold bounding that surface.

We also treat higher dimensional abelian Chern-Simons theory, which exists on manifolds of real dimension $4k+3$. 
Particularly relevant for physics is the Chern-Simons action on 7- and 11-dimensional manifolds. 
Witten \cite{WittenM5effective} has argued that the 7-dimensional abelian Chern-Simons theory is holographically dual to the abelian ``chiral'' two-form, which is a piece of the 6-dimensional $\cN=(2,0)$ superconformal theory. 
Likewise, there is a ``chiral'' four-form that appears in the Type IIB superstring, and is understood as being the holographic dual of 11-dimensional Chern-Simons theory. 
We propose an interpretation of each of these physical situations via factorization algebras that follows the outline above for the ordinary CS/WZW correspondence.
See Section~\ref{sec:physics} for this discussion.

In arbitrary real dimension $4k+3$, the higher dimensional analogs of Chern-Simons
states are sections of interesting vector bundles over the intermediate Jacobians of any complex $2k+1$-fold that admits an oriented null-cobordism.
It would be very interesting to see how a full non-perturbative formulation of our constructions relates back to quadratic refinements of intersection pairings~\cite{HopSing}. 

Another payoff, which we expect to follow from the present work, is a systematic generalization of the role played by the Poisson sigma model in controlling the deformation quantization of Poisson manifolds.
The Poisson sigma model itself (with a Poisson vector space as a target and the half-plane as a source) is an example of the bulk-boundary systems amenable to our methods. 
Here we show that,
in the guise of its stratified factorization algebra, 
we recover the Swiss cheese algebras that play a key role in deformation quantization.
In addition, we discuss a case of Koszul duality arising from a pair of transverse Lagrangians;
it intertwines deformation quantization with the general view on Koszul duality via factorization algebras.

In the near future we expect it to be possible to show that for a boundary condition that is itself topological in nature, 
the bulk factorization algebra is the derived center of the boundary factorization algebra,
just as the bulk observables of the Poisson sigma model are the Hochschild cochains of the deformation quantized algebra living on the boundary.

The setting in which we work --- a certain class of BV theories on manifolds with boundary --- was formulated at the classical level in \cite{ButsonYoo}. 
Their work focuses on a class of interacting perturbative theories, but they do not treat quantization.
The work here takes the first steps of quantization of perturbative bulk-boundary theories in the BV formalism for \emph{free} theories of the type studied in \cite{ButsonYoo}. 
In the PhD dissertation \cite{AR2}, the third named author develops quantization for such interacting theories.

As a remark for readers familiar with the BV-BFV formalism \cite{CMR1, CMR2},
we note that we explore here a less sophisticated situation than a BFV theory on the boundary. 
Here we simply impose a boundary condition --- we force the boundary values of our fields to live in a Lagrangian subspace of all possible boundary values --- rather than work with a Lagrangian foliation.
Note that on a linear symplectic space, picking a linear Lagrangian subspace and picking a linear Lagrangian foliation are in correspondence.
Hence, we hope that our methods, particularly the factorization algebra aspects, may have some role to play in the BV-BFV approach.
In particular, it would be interesting to relate our results to the perspective of \cite{MSW},
who offer a different approach to the CS/WZW correspondence.

\subsection{Outline of the paper}

Section \ref{sec: BBFTs} defines the class of bulk-boundary theories that we study in this paper.
The definition is modeled on the definition of a free BV theory in \cite{CosBook},
but we hope it is transparent to anyone already familiar with the BV formalism in some guise.
We end the section with several examples;
some readers may wish to start there.

Section \ref{sec: FAconstruction} recalls the factorization algebras that appear purely in the bulk or on the boundary,
which were constructed in \cite{CG1}, in various guises.
We then construct the natural factorization algebra for the bulk-boundary system,
modeled on those constructions.
Functional analytic subtleties are addressed in the appendix.

Section \ref{sec: main theorem} states and proves the main theorem,
both for classical and for quantum observables.
Section \ref{sec: apps} addresses specific examples of the theorems.

\subsection{Acknowledgements}

We are lucky to be part of a community bustling with ideas and generous in sharing them.
On the topics connected with this paper---such as free BV quantization, field theories on manifolds with boundary, and generalizations of familiar theories---we are grateful to 
Ben Albert, Dylan Butson, Damien Calaque, Ivan Contreras, Kevin Costello, Chris Elliott, Greg Ginot, Ryan Grady, Andre Henriques, Theo Johnson-Freyd, Si Li, Pavel Safronov, Claudia Scheimbauer, Michele Schiavina, Stephan Stolz, Peter Teichner, Alessandro Valentino, and Philsang Yoo.
We would like to thank Mich\`ele Vergne for comments on an earlier draft of this paper.
The anonymous referee caught several errors and pinpointed ambiguities;
the paper is improved by their eagle eyes. 
Joint time at the Max Planck Institute for Mathematics kickstarted our dialogue about these issues, and we are grateful for the convivial atmosphere and financial support it supplied.
The National Science Foundation supported O.G. through DMS Grant No. 1812049.
E.R. is supported by the National Science Foundation Graduate Research Fellowship Program under Grant No. DGE 1752814. 
B.W. is partially supported by the National Science Foundation Award DMS-1645877.
During the revisions of this paper in spring 2020, all three authors benefited from the hospitality of the Mathematical Sciences Research Institute in Berkeley, California;
While O.G. and B.W. were in residence at MSRI and hence received support from the NSF under Grant No. 1440140.
Any opinions, findings, and conclusions or recommendations expressed in this material are those of the authors and do not necessarily reflect the views of the National Science Foundation.

\section{Free bulk-boundary field theories}
\label{sec: BBFTs}

\subsection{Overview}
In this section, we describe what we mean by a ``free bulk-boundary  system."  
We follow the general discussion with a series of examples of interest to us. 
(The impatient reader should feel free to skip to the examples.)
We  use the Batalin-Vilkovisky (BV) formalism to articulate the relevant definitions, but before we provide full definitions, let us discuss briefly some general principles which underlie the study of field theory on manifolds with boundary. 

Let $M$ be a manifold with boundary, and $\mathring M= M\backslash \del M$ be its interior. 
In Lagrangian field theory, one often starts with a bundle $E\to \mathring M$ and an action functional $S$ that  is a function of the space of sections of $E$, for which we temporarily use the symbol $\sE$. 
The equations of motion for the field theory are the Euler-Lagrange equations for the action $S$. 
One may extend $E$ (and also $\sE$) to $M$, and the equations of motion can also be extended to $M$. 
However, these equations of motion no longer arise from the calculus of variations for $S$ considered as a function on all of $\sE$: 
the argument on $\mathring M$ uses an integration by parts, 
which produces a boundary term when the analogous calculation is carried through on $M$.
A solution to this issue is to restrict $S$ to a subspace of $\sE$ for which the boundary term vanishes. In other words, we \emph{impose boundary conditions on the fields.}
We would like our boundary conditions to be suitably \emph{local}, which means that they are specified by a sub-bundle of the bundle of normal jets $J_\nu(E)$ on $\del M$ (in more coordinate-dependent terms, the boundary conditions impose a point-by-point condition on the values and normal derivatives of sections of $E$ on $\del M$).

Our particular approach to these ideas makes use of the Batalin-Vilkovisky formalism, which is a natural method for encoding the equations of motion and their symmetries in a homotopically coherent way. 
We take as a starting point the definition of a free BV theory given for manifolds without boundary in \cite{CosBook}. 
In the sense of that reference, a free BV theory on $\mathring M$ is defined by a graded vector bundle $E\to \mathring M$, a differential operator $Q$ on $E$ turning $(\sE,Q)$ into an elliptic complex, and a (cohomological degree --1) pairing $\ip_{loc}: E\otimes E\to \text{Dens}_{\mathring M}$. 
The pairing $\ip_{loc}$ induces a pairing $\ip$ on (compactly-supported) sections of $E$ via integration over $\mathring M$.  
The crucial axiom of a free BV theory assumes that $\ip$ is invariant with respect to $Q$. 
For a free BV theory of this sort, define the action functional
\[
S(\phi) = \frac{1}{2}\ip[\phi, Q\phi];
\]
in almost all cases, the invariance of $\ip$ with respect to $Q$ is proved by the same computations which derive the Euler-Lagrange equations of motion $Q\phi = 0$ from $S$. 
Namely, one shows that 
\[
\ip[Q\phi,\phi]_{loc}+(-1)^{|\phi|}\ip[\phi,Q\phi]_{loc}
\]
is a total derivative on $\mathring M$, and so its integral over $\mathring M$ gives an integral over  $\del \mathring M=\emptyset$.
When extending to $M$, therefore, one finds that the failure of the invariance of $\ip$ (under $Q$) to hold is measured by a bilinear pairing on the fields (sections of $E$) which depends only on the values of the fields and their normal jets on the boundary $\del N$. 
In other contexts, this bilinear form is called the Green's form.

In order to have control over the structure of the Green's form, we will introduce in Definition \ref{dfn: tnbft} a particular class of free BV theories whose Green's form has a simple structure. 
This definition is a special case of one introduced in \cite{ButsonYoo}.
Next, we remedy the failure of $\ip$ to be invariant for $Q$ by imposing boundary conditions. We do so in a homologically self-consistent way. In Definition \ref{dfn: bdycondition}, we define the precise nature of the boundary conditions we consider. Finally, in Definition \ref{dfn: bbft}, we fulfill the advertised purpose of this section by defining what we mean by a bulk-boundary field theory.

\subsection{Detailed definitions}

The following definition is a special case of Definition 3.9 of \cite{ButsonYoo}. Throughout the remainder of the text, we fix a manifold $M$ with boundary $\del M$, and we let $\iota: \del M \hookrightarrow M$ denote the inclusion. 

\begin{dfn}
\label{dfn: tnbft}
A free field theory on $M$ that is \textbf{topological normal to the boundary} (free TNBFT) consists of
\begin{itemize}
\item an elliptic complex $(\sE, Q)$ over $M$; here $\sE$ denotes the sheaf of sections of a $\ZZ$-graded smooth vector bundle $E\to M$ of finite total rank, and
\item a skew-symmetric cohomological degree --1 bundle map $\ip_{loc}:E\otimes E\to \Dens_M$
\end{itemize}
satisfying the following conditions:
\begin{enumerate}
\item The pairing $\ip_{loc}$ is fiberwise non-degenerate.
\item If $e_1,e_2\in \sE_{c}$ have compact support contained in $M\backslash \del M$, then
\begin{equation}
\label{eq: invpairing}
\int_M\left(\ip[Qe_1,e_2]_{loc}+(-1)^{|e_1|}\ip[e_1,Qe_2]_{loc}\right)=0,
\end{equation}
i.e., $Q$ is a derivation for the pairing $\ip$ induced from $\ip_{loc}$ by integration over $M$.
\item In a tubular neighborhood $T\cong \del M\times [0,\epsilon)$ of $\del M$, there is an isomorphism
\begin{equation}
\label{eq: decomp}
E\big|_{T} \cong \Eb \boxtimes \Lambda^\bullet (T^* [0,\epsilon)),
\end{equation}
where $\Eb$ is a graded vector bundle over $\del M$. 
With respect to this isomorphism, we require that
\begin{itemize}
\item In the tubular neighborhood $T$ where we make the identification of Equation \eqref{eq: decomp}, $Q$ have the form $\Qb\otimes 1+1\otimes d_{dR}$, where $\Qb$ gives $\sEb$ (the sheaf of sections of $\Eb$ on $\del M$) the structure of an elliptic complex on $\del M$ and $d_{dR}$ is the de Rham differential in the normal direction, and
\item the pairing $\ip_{loc}$ have the form $\ip_{loc,\del} \boxtimes \wedge$, where $\ip_{loc,\del}$ is a vector bundle map
\[
\ip_{loc,\del} : \Eb\otimes \Eb\to \Dens_{\del M}
\]
on $\del M$ which is fiberwise non-degenerate, of cohomological degree 0, skew-symmetric, and satisfies
\[
\int_{\del M} \left(\ip[\Qb e'_1,e'_2]_{loc, \del}+(-1)^{|e'_1|}\ip[e'_1,\Qb e'_2]_{loc, \del}\right)=0
\]
for all compactly supported sections $e'_1, e'_2$ of $\Eb$. 
\end{itemize}
\end{enumerate}
We will often use the letter $\sE$ to denote the full information of the TNBFT $(\sE,Q,\sEb,\Qb,\ip_{loc},\ip_{loc, \del})$.
\end{dfn}

Following the discussion in the previous subsection, 
we note that the pairing $\ip_{loc, \del}$ is essentially the datum of the Green's form.
There should be convenient generalizations of this setup that do not require the elliptic complex $\sE$ to be topological in the normal direction,
but these require more sophisticated analysis.

\begin{rmk}
Having a manifold with boundary is not essential here.
One can make a similar definition 
if there is a hypersurface $S$ in $M$ such that in a tubular neighborhood of $S$, 
the field theory has an analogous decomposition as an elliptic complex along $S$ tensored with the de Rham complex in the normal direction. 
(This setup is reminiscent, in Lorentzian field theories, of picking a foliation of a globally hyperbolic manifold by spacelike hypersurfaces.)
This more general situation would enable one to study certain domain walls in the BV context. 
Since we are only interested in boundary conditions, however, we do not explore this more general definition.
\end{rmk}

\begin{ntn}
There is a sheaf map $\rho: \sE\to \iota_*\sEb$ that is the composite of restriction to the tubular neighborhood $T$, followed by the isomorphism 
\[
\sE\big|_T\cong \sEb\,\hotimes\, \Omega^\bullet_{[0,\epsilon)},
\]
followed by the evaluation map 
\[
\sEb\,\hotimes\, \Omega^\bullet_{[0,\epsilon)}\to \sEb
\]
induced from the pullback of forms to $t=0$. The map $\rho$ is a cochain map.
We denote by $\ip$ the pairing between sections $e_1,e_2\in \sE$ (at least one of which has compact support) given by 
\[
\ip[e_1,e_2]=\int_M \ip[e_1,e_2]_{loc},
\]
and similarly for $\ip_{\del}$. 
\end{ntn}

\begin{rmk}
\label{rmk: action}
To make contact with Lagrangian field theory, we note that the pairing $\ip_{loc}$ and the differential $Q$ give rise to the action functional
\[
S(\phi)=\frac{1}{2}\ip[\phi,Q\phi].
\]
We will see below (e.g., Equation \ref{eq: noninvpairing}) that the data $(\sEb,\Qb,\ip_{\del})$ encode the boundary terms that arise from variational calculus. Our formulation of the problem guarantees that this construction is done in a way consistent with gauge symmetry on the boundary. 
\end{rmk}

\begin{rmk}
Condition (3) explains why TNBFTs are considered ``topological normal to the boundary'': 
a solution to the equations of motion is locally constant in the direction normal to the boundary,
as the fields in the normal direction are entirely dictated by the behavior of de Rham forms in that direction. 
\end{rmk}

Equation \ref{eq: invpairing} does not need to hold for sections $e_1,e_2$ that have non-zero values at the boundary. In fact, we find that, in general,
\begin{equation}
\label{eq: noninvpairing}
\ip[Qe_1,e_2]+(-1)^{|e_1|}\ip[e_1,Qe_2]= \ip[\rho e_1,\rho e_2]_{\del}.
\end{equation}
In other words, the pairing $\ip_{\del}$ on $\sEb$ measures the failure of $\ip$ to be invariant for the differential $Q$. Because of Equation \eqref{eq: noninvpairing}, a free TNBFT is not, strictly speaking, a field theory on $M$; however, a free TNBFT is still a field theory on $M\setminus \del M$. We will therefore persist in the usage of the term ``field theory'' for free TNBFTs, even when we consider them on the whole spacetime manifold $M$.

The pairing $\ip_{\del}$ is also closely related to the boundary terms which arise when integrating by parts in the derivation of the Euler-Lagrange equations of motion for the action of Remark \ref{rmk: action}. In order to construct the quantum observables in the Batalin-Vilkovisky (BV) formalism and to avoid these boundary terms, we will need Equation \ref{eq: invpairing} to hold for \emph{a broader class of} $e_1,e_2$ than the class of sections in $\sE$ that vanish on $\del M$. To remedy this, we introduce a notion of \emph{boundary condition}.

\begin{dfn}
\label{dfn: bdycondition}
A \textbf{local Lagrangian boundary condition} for a free TNBFT $\sE$ is a graded subbundle $L\to \partial M$ of $\Eb \to \partial M$ with the following four properties:
\begin{itemize}
\item the total rank of $L$ is half that of~$\Eb$,
\item $\ip_{loc, \del}$ is identically zero on $L\otimes L$, and
\item the sheaf $\sL$ of smooth sections of $L$ on $\partial M$ is a subcomplex of $\sEb$ with respect to the differential~$\Qb$.
\item The complex $(\sL,\Qb)$ is elliptic.
\end{itemize}
\end{dfn}

The following is the main definition of this section.

\begin{dfn}
\label{dfn: bbft}
Given a free TNBFT $\sE$ and a boundary condition $\sL$ for $\sE$, we will call the pair $(\sE,\sL)$ a \textbf{free bulk-boundary field theory}. For a free bulk-boundary field theory $(\sE,\sL)$, we denote by $\condfields$ the pullback of sheaves of complexes
\[
\begin{tikzcd}
\condfields\arrow[r]\arrow[d] \arrow[rd, phantom, "\lrcorner", at start] & \sE\arrow[d,"\rho"]\\
\iota_*\sL \arrow[r,hook] & \iota_*\sEb
\end{tikzcd}.
\]
In other words, $\condfields(U)$ consists of sections $e\in \sE(U)$ such that $\rho(e)\in \iota_*\sL(U)\subset \iota_*\sEb(U)$. We will call $\condfields$ the \textbf{sheaf of fields of the bulk-boundary system.} The fields in $\condfields$ satisfy a boundary condition imposed by the choice $\sL$. 
\end{dfn}

\begin{rmk}
The term \textit{bulk-boundary field theory} deserves to encompass a much larger class of situations,
including those where the equations of motion are not locally constant in the normal direction to the boundary,
but that is the only situation in which we work in this paper.
Hence we use the term here as shorthand.
We will also use the term ``bulk-boundary system'' to denote the same object.
\end{rmk}

\begin{rmk}
Note that all maps in the pullback square defining $\condfields$ are cochain maps, so the differential $Q$ on $\sE$ descends to one on $\condfields$. Since $\condfields$ is a subsheaf of $\sE$, one can also restrict the pairing $\ip$ to $\condfields$. Then, it is straightforward to verify (using the definitions directly) that Equation \ref{eq: invpairing} is satisfied for the fields of the bulk-boundary system.
Hence, we are free to think of $(\condfields, Q)$ as a bulk-boundary free BV theory.
In a sense, $\condfields$ is a maximal subspace of $\sE$ for which Equation \ref{eq: invpairing} is satisfied. We will find that most of the constructions of \cite{CG1} for the analogous case with $\del M =\emptyset$ carry over with little or no change once we use $\condfields$ for the space of fields.
\end{rmk}

\begin{rmk}
We note that, since the map $\rho$ is an epimorphism, $\condfields$ also coincides with the \emph{homotopy} pullback $\sE\times^h_{(\iota^* \sEb)}\iota_* \sL$ in a suitable model category of presheaves of complexes (see \cite{AR2} for more details). Hence, $\condfields$ imposes the boundary condition $\sL$ in a homotopically consistent way. Physically, the way we impose boundary conditions guarantees that the gauge symmetries of the theory remain manifest. In \cite{edgemodesyangmills}, a similar procedure is performed for abelian Yang-Mills theory. There, the authors also take care to impose boundary conditions in a homologically consistent way.
\end{rmk}

\subsection{Examples of free bulk-boundary systems}
\label{sec: examples}

We now discuss several examples of free bulk-boundary systems.

\begin{example}
\label{ex: toplmech}
Suppose $V$ is a symplectic vector space with symplectic form $\omega$. Let $M=[0,\epsilon)$ and $\sE=\Omega^\bullet_{[0,\epsilon)}\otimes V$, together with the pairing $\ip$ induced from the Poincar\'e duality pairing and $\omega$. Here, $\Eb=\sEb=V$, and $\ip_{loc, \del}=\omega$. This theory is \textbf{topological mechanics}. A Lagrangian subspace $L$ of $V$ gives a boundary condition for topological mechanics.
\end{example}

\begin{example}
\label{ex: psm}
Let $\Sigma$ be any surface with boundary, and let $V$ be a vector space with a constant Poisson structure, i.e., $V$ is a vector space equipped with a skew-symmetric map $\Pi: V^\vee \to V$. 
Let 
\[(\sE,Q)=( \Omega^\bullet_{\Sigma}\otimes V\oplus \Omega^\bullet_{\Sigma}\otimes V^\vee[1],\d_{dR}\otimes 1+1\otimes \Pi).\]
The pairing $\ip_{loc}$ is defined using the wedge product and the natural pairing between $V^\vee$ and $V$. It is evident that one can write
\[
(\sEb,\Qb)=(\Omega^\bullet_{\del \Sigma}\otimes V\oplus \Omega^\bullet_{\del \Sigma}\otimes V^\vee[1],\d_{dR}\otimes 1+1\otimes \Pi),
\]
and $\ip_{loc, \del}$ is again defined using the wedge product of forms and the canonical pairing between $V^\vee$ and $V$. This theory is a special case of the \textbf{Poisson sigma model} \cite{MR1854134}. The subcomplex $\Omega^\bullet_{\del \Sigma} \otimes V\subset \sEb$ gives a boundary condition for this theory. 
\end{example}

\begin{example}
Suppose $\fA$ is a complex vector space together with a non-degenerate symmetric bilinear pairing $\kappa$. Let $M$ be an oriented 3-manifold with boundary. For $(\sE,Q)$ we take $(\Omega^\bullet_{M}\otimes \fA[1],d_{dR})$. For the pairing $\ip_{loc}$ we take 
\[
\ip[\alpha,\alpha']_{loc} = \kappa(\alpha, \alpha'),
\]
where we are implicitly taking a wedge product of forms and only keeping the top-form component of the resulting wedge product. From these characterizations, it is evident that $(\sEb,\Qb) =(\Omega^\bullet_{\del M}\otimes \fA[1],d_{dR})$, and 
\[
\ip[\alpha,\alpha']_{loc, \del}=\kappa(\alpha, \alpha').
\]
This theory is an \textbf{abelian Chern-Simons theory}. 
In the bulk $3$-manifold, $M \setminus \partial M$, this elliptic complex is simply abelian Chern-Simons theory where we view $\fA$ as an abelian Lie algebra.
The solutions to the bulk equations of motion are the $\fA$-valued closed one-forms. 

If $M=\Sigma\times \RRge$ (where $\Sigma$ is a Riemann surface), the space of fields is endowed with the decomposition
\[
\sE=\Omega^{0,\bullet}_{\Sigma}\,\hotimes\, \Omega^\bullet_{\RRge}\otimes \fA[1]\oplus \Omega^{1,\bullet}_{\Sigma}\,\hotimes\, \Omega^\bullet_{\RRge}\otimes\fA,
\]
with differential $Q=\del+\delbar+d_{dR}$.

The boundary condition we consider depends on the choice of a complex structure on the boundary $\del M$.
Henceforth, when we want to stress the dependence on the complex structure, we denote the boundary Riemann surface by~$\Sigma$.

Given a holomorphic vector bundle $V$ on $\Sigma$, there is a resolution for its sheaf of holomorphic sections $\sV^{hol}$ given by the Dolbeault complex $\left(\Omega^{0,\bullet}(\Sigma, V), \dbar\right)$. 
The differential is the Dolbeault operator $\dbar : \Omega^{0}(\Sigma, V) \to \Omega^{0,1}(\Sigma, V) = \Gamma(T^{*0,1} \otimes V)$ defining the complex structure on $V$. 
In the case that $V = T^{*1,0}$, we denote this Dolbeault complex by $\Omega^{1,\bullet}(\Sigma)$ with the $\dbar$-operator understood. 

The subcomplex $\Omega^{1,\bullet}_\Sigma\otimes \fA\subset \Omega^\bullet_\Sigma\otimes \fA[1]$ defines a boundary condition for abelian Chern-Simons theory (at any level if $M=\CC\times \RRge$). 
To see this, consider $\sL$ as the sections of a vector bundle $L$ on $\Sigma$.
It is clear that the rank of $L$ is half that of $\Eb$, where $\Eb$ is the vector bundle whose sheaf of sections is $\sE_{\partial}$. 
Also, $\ip_{loc, \del}$ is identically zero on $L \otimes L$ since only forms of type $(1,\bullet)$ appear in $\sL$. 
Finally, the cochain complex $ \Omega^{1,\bullet}_\Sigma\otimes \fA$ is a subcomplex of the full de Rham complex since $\dbar \alpha = \d \alpha$ for forms $\alpha$ of type $(1,\bullet)$. 
We call it the \textbf{chiral WZW boundary condition}. Notice that although Chern-Simons theory is topological, we may choose a non-topological boundary condition for the theory. In this situation, the boundary condition has a chiral, or holomorphic, nature.

On $M=\Sigma \times \RR_{\geq 0}$ we remark on a slightly different presentation of abelian Chern--Simons theory, as a deformation of BF theory.
Using $\kappa$, we can identify the fields $\sE$ with
\[
\Omega^{0,\bullet}_{\Sigma}\,\hotimes\, \Omega^\bullet_{\RRge}\otimes \fA[1]\oplus \Omega^{1,\bullet}_{\Sigma}\,\hotimes\, \Omega^\bullet_{\RRge}\otimes\fA^*,
\]
by keeping track of the Dolbeault decomposition along $\Sigma$.
In this decomposition, fields are pairs $(\alpha,\beta)$, where $\beta$ now takes values in $\fA^*$. 
Upon making this identification the pairing becomes $\ip[\alpha, \beta]_{loc} = \alpha \beta$ and the differential is $\kappa \partial + \dbar + d_{dR}$. 
In this presentation, it now makes sense to contemplate the limit $\kappa \to 0$, 
which results in a BF theory on $\Sigma \times \RR_{\geq 0}$ that is partially topological. 
We dub this boundary condition $\sL = \Omega^{1,\bullet}_\Sigma\otimes \fA^*$, the chiral WZW boundary condition. 
We note that the chiral WZW boundary condition is indeed a boundary condition for Chern-Simons theory for any $\kappa$ (including $\kappa=0$).
We remark that even after quantization, there will be no shift by the critical level, since we only consider free theories (i.e., abelian WZW theories).
\end{example}

\begin{example}
\label{ex: highercs}
Let $M$ be an oriented manifold of dimension $4k+3$, and suppose $\del M$ has the structure of a complex manifold (of dimension $2k+1$). 
As in the case of ordinary Chern--Simons theory, we fix a vector space $\fA$ equipped with a non-degenerate symmetric pairing $\kappa$. 

The fields are the (shifted) de Rham forms
\[
\sE = \Omega^\bullet_M \otimes \fA [2k+1], 
\] 
with $Q=\d_{dR}$.
There is an obvious degree $(-1)$-pairing on the fields given by wedging and integrating. 
This theory describes \textbf{higher-dimensional abelian Chern-Simons theory}, the action functional reads 
\[
\int_M \kappa(\alpha, \d \alpha)
\]
where $\alpha \in \Omega^\bu_M \otimes \fA[2k+1]$.

Geometrically, this theory encoding deformations of the trivial flat $U(1)$ $k$-gerbe with fiber $\fA$.
(Taking $k=0$, we note that a flat $U(1)$ 0-gerbe is a flat $U(1)$-bundle.)
We have $\sEb=\Omega^\bullet_{\del M}\otimes \fA[2k+1]$, and $\Qb= \d_{dR}$. 
The pairings are defined exactly as in the previous example, by wedging and integration. 

Like the WZW boundary condition for ordinary Chern--Simons, there is a boundary condition which utilizes the complex structure of the boundary $\partial M$. 
It is given by the Lagrangian
\[
\sL = \Omega^{> k, \bullet}_{\partial M} \otimes \fA [k],
\] 
which we call the \textbf{intermediate Jacobian} boundary condition, due to it being a piece of the Hodge filtration.
We return to this example in more detail in Section~\ref{sec: higherCS}. 
\end{example}

\begin{example}
\label{ex: Riemhighercs}
There is an alternative boundary condition of higher dimensional Chern--Simons that depends on a Riemannian metric  rather than a complex structure. 
As above, let $M$ be an oriented manifold of dimension $4k+3$, and suppose the boundary $N = \partial M$ is equipped with a Riemannian structure. 
In turn, we decompose the middle de Rham forms on $N$ into the $\pm \sqrt{-1}$-eigenspaces
\begin{equation}\label{eqn:decomp}
\Omega^{2k+1}(N) = \Omega^{2k+1}_+ (N) \oplus \Omega^{2k+1}_-(N) 
\end{equation}
of the Hodge star operator. 

Consider the subcomplex of~$\Omega^\bu_{\partial M} \otimes \fA [2k+1]$:
\[
\sL = \bigg(\Omega^{2k+1}_+(N) \otimes \fA \xrightarrow{\d} \Omega^{2k+2} (N) \otimes \fA [-1] \xrightarrow{\d} \cdots \xrightarrow{\d} \Omega^{4k+2}(N) \otimes \fA [-2k-1]\bigg) .
\]
It defines a boundary condition for $(4k+3)$-dimensional abelian Chern-Simons theory. 

The elliptic complex on $N$ perpendicular to the boundary condition $\sL$ is $\sL^\perp$, can be identified with
\begin{equation}\label{eqn:dminus}
\sL^\perp = \bigg( \Omega^{0} (N) \otimes \fA [2k+1].\xto{\d} \Omega^1(N) \otimes \fA[2k] \to \cdots \to \Omega^{2k} (N) \otimes \fA [1] \xto{\d_-} \Omega^{2k+1}_- (N) \otimes \fA \bigg)
\end{equation}
where $\d_- : \Omega^{2k} (N) \to \Omega^{2k+1}_- (N)$ denotes the de Rham differential followed by the projection using the decomposition~(\ref{eqn:decomp}).
\end{example}

\section{The factorization algebras at play}
\label{sec: FAconstruction}

In this section we describe the three factorization algebras that appear in a bulk-boundary system:
\begin{itemize}
\item the observables $\Obs_\sE$ living purely in the bulk $\mathring M$, which depend only on the BV theory in the bulk,
\item the observables $\Obs_{\sL}$ of the boundary condition, which live only on the boundary $\partial M$, and
\item the observables $\Obs_{\sE,\sL}$ of the bulk-boundary system, which lives on the whole manifold $M$ with boundary.
\end{itemize}
There are classical and quantum versions of both factorization algebras.
Now aware of the these three algebras, 
the reader can skip to Section~\ref{sec: main theorem} and understand the statement of our main theorems.

The bulk observables $\Obs_\sE$ arising here were defined in \cite{CG1}, and they are a straightforward interpretation of the observables in a free BV theory.
The observables of the boundary condition $\Obs_{\sL}$ are defined in a similar way.
At the classical level, they are simply functions on the space $\sL$, but  the quantization uses a Poisson structure arising from the map to $\sEb$ that identifies $\sL$ as a Lagrangian in $\sEb$.
In this sense, the boundary condition behaves like a Poisson field theory,
in contrast to the symplectic-type bulk theory.

The observables $\Obs_{\sE,\sL}$ are constructed in an analogous way to the other algebras.
The classical observables realize, in a homotopical sense, the algebra of functions on the space of solutions to the equations of motion that satisfy the boundary condition.
The quantization is in the spirit of the BV formalism; 
it amounts to changing the differential by adding an operator determined by the natural pairing on the fields, with boundary condition imposed.
Our main theorems show that $\Obs_{\sE,\sL}$ interpolates between $\Obs_\sE$ and~$\Obs_{\sE,\sL}$,
and in this way we see that there is a natural quantization of the bulk-boundary system that realizes a correspondence between the bulk and boundary systems themselves.

\subsection{Bulk observables}
\label{sec: bulk obs}

Chapter 4 of \cite{CG1} is devoted to constructing and analyzing the observables, both classical and quantum, of a free BV theory on a smooth manifold.
Here we simply recall the definitions.

\begin{dfn}
Let $\sE$ be a free TNBFT. 
The factorization algebra of \textbf{classical observables for $\sE$} assigns to an open subset $U\subset \mathring M$ the (differentiable) cochain complex
\[
(\Sym(\sE_c[1](U)),Q)=:\Obcl_{\sE}(U),
\]
where the symmetric powers are taken with respect to the completed bornological tensor product of convenient vector spaces (see, e.g. Definition B.4.9 and Section B.5.2 of~\cite{CG1}).
\end{dfn}

Note that for a smooth vector bundle $V\to M$, these completed tensor products can be understood concretely as
\[
(\sV_{c}(U))^{\hotimes k}\cong C^\infty_{c}(U^{\times k}; V^{\boxtimes k}). 
\]
In other words, they are the compactly supported sections on the $k$-fold product $U^k$ with values in the natural vector bundle $V^{\boxtimes k} \to U^k$.

Something a bit subtle is happening in this definition. 
{\it A priori} the classical observables ought to consist of functions on the fields $\sE$;
in other words, they ought to be a symmetric algebra on the linear dual vector space or, better yet, the continuous linear dual.
Here, however, we took a symmetric algebra on $\sE_c[1]$,
which looks different.
Two facts combine to explain our choice.
First, the local pairing lets us identify the continuous linear dual of $\sE$ with the distributional and compactly supported sections of $E[1] \to M$:
every such section determines a linear functional on $\sE$ by plugging it into the pairing.
Second, the Atiyah-Bott lemma (see Appendix E of \cite{CG1}) shows that the elliptic complex of distributional, compactly supported sections of $E[1] \to M$ is continuously quasi-isomorphic to the subcomplex of smooth, compactly supported sections of $E[1] \to M$.
Together, these facts show that our definition captures correctly --- up to quasi-isomorphism --- the most natural choice of classical observables.
Concretely, we are working with {\em smeared} observables.

With our definition, BV quantization is straightforward,
because the pairing determines a natural BV Laplacian $\Delta: \Sym(\sE_c[1](U)) \to \Sym(\sE_c[1](U))$ as follows.
We set $\Delta = 0$ on the constant and linear terms (i.e., the subspace $\Sym^{\leq 1}(\sE_c[1](U)$), 
and we require
\[
\Delta(a b) = \Delta(a) b + (-1)^{|a|} a \Delta(b) + \{a,b\}
\]
for arbitrary $a$ and $b$. 
Here, $\{\cdot,\cdot\}$ is the unique biderivation (with respect to the product in the symmetric algebra) on 
\[
\Sym(\sE_c[1](U))\times \Sym(\sE_c[1](U))
\]
which coincides with $\ip[\cdot,\cdot]$ on 
\[
\sE_c[1](U)\times \sE_c[1](U).
\]
This equation defines $\Delta$ inductively on the symmetric powers. 

For instance, if $a$ and $b$ are linear, then $ab \in \Sym^2(\sE_c[1](U))$, 
and we see that
\[
\Delta(a b) = \ip[a,b]
\]
because we have set $\Delta(a) = 0 = \Delta(b)$. 
Such pure products $ab$ span $\Sym^2(\sE_c[1](U))$, 
so we have defined $\Delta$ on all quadratic functionals.
To determine $\Delta$ on $\Sym^3(\sE_c[1](U)$, we use the equation and our knowledge of $\Delta$ on $\Sym^{\leq 2}(\sE_c[1](U))$;
inductively continue this process to higher symmetric powers.

By construction, $\Delta$ is a second-order differential operator on the graded commutative algebra $\Sym(\sE_c[1](U))$.
It is straightforward to verify that $\Delta^2 = 0$ and that $\Delta$ commutes with $Q$ (because $Q$ is compatible with the pairing $\ip$).
Hence we posit the next definition, following the BV formalism.

\begin{dfn}
Let $\sE$ be a free TNBFT. 
The factorization algebra of \textbf{quantum observables for $\sE$} assigns to an open subset $U\subset \mathring M$, the (differentiable) cochain complex
\[
(\Sym(\sE_c[1](U))[\hbar],Q+\hbar \Delta)=:\Obq_{\sE}(U),
\]
where the symmetric powers are taken with respect to the completed bornological tensor product of convenient vector spaces.
\end{dfn}

\subsection{Observables of the boundary condition}

A boundary condition $\sL$ leads to factorization algebras on the boundary in a parallel fashion.

At the classical level, the idea is that we want to use a commutative algebra of functions on $\sL$,
which we take to be a symmetric algebra on the continuous linear dual $\sL^*$.
It is convenient to work with a smeared (and hence smooth) version of $\sL^*$.
One approach is to note that $\sL$ is a subspace of $\sEb$, 
and so we could work with the quotient of $\Obcl_{\sEb}$ by the ideal of functions that vanish on the subspace $\sL$.
This approach is canonically determined by the map $\sL \to \sEb$, and hence manifestly meaningful.
On the other hand, it is convenient to have an explicit graded vector bundle to use, 
particularly when we quantize and need to transport the BV Laplacian for the bulk theory to an operator on the boundary observables.
Hence we now introduce a different approach that we will see, later, is equivalent.

\begin{constr}
Let $\sL$ be a boundary condition for a free TNBFT associated to the graded subbundle $L$ of $\Eb$. 
Let $L^\perp$ be a complementary subbundle so that $\Eb = L \oplus L^\perp$. 
Let $\sL^\perp$ denote the sheaf of smooth sections of $L^\perp$,
and let $\sL^\perp_c$ the cosheaf of compactly supported smooth sections of~$L^\perp$.
With respect to this splitting, the differential $\Qb$ decomposes as $Q_{L}+Q_{L^\perp}+Q_{rel}$, 
where $Q_{L}$ preserves $\sL$, $Q_{L^\perp}$ preserves $\sL^\perp$, and $Q_{rel}$ maps $\sL^\perp$ to $\sL$. 
(There is no operator from $\sL$ to $\sL^\perp$ because we have assumed that $\Qb$ preserves~$\sL$.)  
\end{constr}

Notice that every element of $\sL^\perp_c$ determines a continuous linear functional on $\sL$ via the local pairing $\ip_{loc, \del}$ on $\sE_\partial$.
In fact, these smeared observables encompass essentially all the linear functionals:
by the Atiyah-Bott lemma, the complex $(\sL^\perp_c, Q_{L^\perp})$ is continuously quasi-isomorphic to the complex of compactly supported distributional sections of $\Eb/L$ with the differential induced by $\Qb$.
Hence a symmetric algebra on $\sL^\perp_c$ deserves to be understood as an algebra of observables.

\begin{dfn}
\label{dfn: classical currents}
Let $\sL$ be a boundary condition for a free TNBFT.
The factorization algebra of \textbf{classical boundary observables for $\sL$} assigns to an open subset $U\subset \partial M$, 
the (differentiable) cochain complex
\[
(\Sym(\sL^\perp_c(U)),Q_{L^\perp})=:\Obcl_{\sL}(U),
\]
where the symmetric powers are taken with respect to the completed bornological tensor product of convenient vector spaces.
\end{dfn}

At the quantum level, one obtains a Heisenberg-type deformation of $\Obcl_\sL$ as a factorization algebra.
The relevant deformation arises from a canonical bilinear form on $\sL^\perp_c$ determined by our construction.
Let $\mu$ be the following local degree $-1$ cocycle on $\sL^\perp_c$:
for any pair of compactly-supported sections $e_1$ and $e_2$ on an open $U\subset \del M$, define
\begin{equation}
\label{eq: twistcocycle}
\mu(e_1,e_2)=-\int_{\del M}\ip[e_1,Q_{rel} e_2]_{loc,\del}.
\end{equation}
The seemingly strange choice of sign is for a convenient statement of our main result, Theorem \ref{thm: maingenlq}.
We use this pairing to define a second-order differential operator $\hbar \Delta_\mu$ on $\Sym(\sL^\perp_c(U))[\hbar]$ of cohomological degree 1,
just as we constructed the BV Laplacian $\Delta$ on the bulk observables.

\begin{dfn}
\label{dfn: quantum currents}
Let $\sL$ be a boundary condition for a free TNBFT.
The factorization algebra of \textbf{quantum boundary observables for $\sL$} assigns to an open subset $U\subset \partial M$, 
the (differentiable) cochain complex
\[
(\Sym(\sL^\perp_c(U))[\hbar],Q_{L^\perp} + \hbar \Delta_\mu)=:\Obq_{\sL}(U),
\]
where the symmetric powers are taken with respect to the completed bornological tensor product of convenient vector spaces.
\end{dfn}

\begin{rmk}
The quotient map $q_L: \Eb \to \Eb/L$ makes $L^\perp$ canonically isomorphic to the quotient bundle $\Eb/L$,
and hence we can identify $L^\perp$ with the image of a splitting of that quotient map.
Any two choices of splitting $L^\perp_0$ and $L^\perp_1$ are related by a bundle automorphism of $\Eb$.
We emphasize this isomorphism is at the point set level; 
it is an automorphism of graded vector bundles.
Using this automorphism one gets a natural equivalence between the associated pairings $\mu_0$ and~$\mu_1$.
Hence, any two versions of the construction above are isomorphic.
\end{rmk}

\begin{rmk}
The pairing $\mu$ determines a central extension $\widehat{\sL^\perp}_c(U)$ of $\sL^\perp_c(U)$ (as an abelian dg Lie algebra) by the vector space $\CC\hbar$ placed in degree~1.
That is, $\widehat{\sL^\perp}_c(U)$ is a kind of Heisenberg Lie algebra.
As it is defined for any open subset $U$ of the whole manifold $\del M$, 
we get a precosheaf of central extensions.
The quantum observables are then the Chevalley-Eilenberg chains of this dg Lie algebra $\widehat{\sL^\perp}_c$.
Thus our definition above is a case of taking a twisted enveloping factorization algebra. 
See Definition 3.6.4 of \cite{CG1} for an extensive discussion,
and Chapter 4 for an explanation of why this construction encodes canonical quantization.
\end{rmk}

\subsection{Observables of the bulk-boundary system}

There is a natural way to extend our methods above to obtain observables on $\sE_\sL$, 
which describes solutions to the equations of motion for fields in $\sE$ that must live in $\sL$ on the boundary.
We will begin by describing the corresponding functor
\[
\mathrm{Opens}(M)\to \mathrm{Ch}
\]
and then turn to verifying it is a factorization algebra.

\begin{dfn}
Let $(\sE,\sL)$ be a free bulk-boundary field theory. 
The prefactorization algebra of \textbf{bulk-boundary classical observables for $(\sE,\sL)$} assigns to each open subset $U\subset M$, 
the (differentiable) cochain complex
\[
(\Sym(\condfieldscs[1](U)),Q)=:\Obcl_{\sE,\sL}(U),
\]
where $\condfieldscs$ denotes the cosheaf of compactly-supported fields for the bulk-boundary system
(i.e., elements of $\sE_{\sL}(U)$ whose support is compact).
The symmetric powers are taken with respect to the completed bornological tensor product of convenient vector spaces.
\end{dfn}

To see that $\Obcl_{\sE,\sL}$ is a prefactorization algebra, 
one can borrow verbatim Section 3.6 of~\cite{CG1}.

\begin{rmk}
We note here that in the appendices, 
we provide two useful results,
\begin{itemize}
\item a more geometric interpretation of the tensor powers $\sE_{\sL}(U)^{\hotimes k}$ and
\item a version of the Atiyah-Bott lemma for the bulk-boundary fields (cf. Appendix D,~\cite{CG1}),
\end{itemize}
that underpin our choice of smeared observables for the bulk-boundary system.
Analogs of these results played a key role in the case of free BV theories on manifolds without boundary.
The first allows us to recognize why the completed bornological tensor product is natural here,
and it also plays a role in the proof that we get a factorization algebra.
The second justifies that working with the continuous linear dual $\sE_{\sL}(U)^\vee$ adds no further information than $\sE_{\sL,c}(U)[1]$, 
up to continuous quasi-isomorphism.
\end{rmk}

In fact, we can, without much difficulty, show that the classical observables form a factorization algebra, that is, they satisfy the local-to-global condition of Definition 6.1.4 in~\cite{CG1}. 

\begin{thm}
\label{thm: obcl}
For a free bulk-boundary theory $\sE$ with local Lagrangian boundary condition $\sL$, 
the classical observables $\Obcl_{\sE,\sL}$ form a factorization algebra.
\end{thm}

\begin{proof}
The context here is nearly identical to that of Theorem 6.5.3(ii) of \cite{CG1}. By the same arguments as in the proof of that theorem, we need only to show that, given any Weiss cover $\fU=\{U_i\}_{i\in I}$ of an open subset $U\subset M$, the map 
\begin{equation}
\label{eq: cechmap}
\bigoplus_{n=0}^\infty \bigoplus_{i_1,\cdots, i_n}\Sym^m\left( \condfieldscs[1](U_{i_1}\cap\cdots\cap U_{i_n})\right) [n-1]\to \Sym^m (\condfieldscs[1](U))
\end{equation}
is a quasi-isomorphism, where the left-hand side is endowed with the \v{C}ech differential. According to the appendix, particularly Corollary~\ref{crl: borntensorofdirichlet}, 
\[
\Sym^m\left( \condfieldscs[1](U_{i_1}\cap\cdots\cap U_{i_n})\right)
\]
is the subspace of 
\[
\Sym^m\left( \sE_{c}[1](U_{i_1}\cap\cdots\cap U_{i_n})\right)\subset C^\infty_{c}((U_{i_1}\cap \cdots \cap U_{i_n})^m,(E[1])^{\boxtimes m})
\]
consisting of those sections that lie in $(L\oplus \Eb \d t)_{x_1}\otimes E_{x_2}\otimes\cdots \otimes E_{x_m}$ whenever the first of the points $x_1,\cdots, x_m\in (U_{i_1}\cap \cdots \cap U_{i_n})^m$ lies on $\del M$, and similarly for $x_2, \cdots, x_m$. The proof of Lemma A.5.7 of \cite{CG1} constructs a contracting homotopy of the mapping cone of Equation \ref{eq: cechmap} without any conditions imposed at the boundary of $M$. Because the contracting homotopy involves only multiplication by smooth functions and addition of sections, it preserves the lie-in condition for $(\condfieldscs)^{\hotimes m}$. Hence, the contracting homotopy from the proof of Lemma A.5.7 of \cite{CG1} gives a contracting homotopy for the mapping cone of Equation~\ref{eq: cechmap}, so that the map of that equation is a quasi-isomorphism.
\end{proof}

We also define a factorization algebra of quantum observables.

\begin{dfn}
Let $(\sE,\sL)$ be a free bulk-boundary field theory. 
The prefactorization algebra of \textbf{bulk-boundary quantum observables for $(\sE,\sL)$} assigns to each open subset $U\subset M$, 
the (differentiable) cochain complex
\[
(\Sym(\condfieldscs[1](U)[\hbar],Q+\hbar \Delta)=:\Obq_{\sE,\sL}(U).
\]
Here $\Delta$ is the restriction of the BV Laplacian for $\Obq_\sE$ to this graded subspace. 
\end{dfn}

\begin{rmk}
The fact that $Q+\hbar\Delta$ is a differential on $\Obq_{\sE,\sL}(U)$ requires some proof. 
In the case where $\del M$ is empty, it follows from the invariance of $\ip$ under $Q$ (see equation \ref{eq: invpairing}). 
In the present case, equation \ref{eq: invpairing} is satisfied for $\condfields$, 
so that $Q+\hbar \Delta$ squares to zero on $\Obq_{\sE,\sL}(U)$. 
This property motivates the use of local Lagrangian boundary conditions for TNBFTs.
\end{rmk}

\begin{thm}
The functor $\Obq_{\sE,\sL}$ is a factorization algebra. 
\end{thm}

\begin{proof}
That $\Obq_{\sE,\sL}$ is a prefactorization algebra is an immediate consequence of the fact that $\Obcl_{\sE,\sL}$ is, since the BV Laplacian is local. 
To see that the local-to-global condition is also satisfied, note that $\Obq_{\sE,\sL}(U)$ has a filtration given by 
\[
F^n\Obq_{\sE,\sL}(U) = \bigoplus_{j+k \leq n} \hbar^j \Sym^k(\sE_{\sL,c}(U)[1])
\]
for every open subset $U$.
The differential on $\Obq_{\sE,\sL}(U)$ preserves this filtration. Moreover, for any Weiss cover $\fU$ of $U$, the \v{C}ech complex $\check{C}(\fU, \Obq_{\sE,\sL})$ for this cover also has a filtration and the map 
\begin{equation}
\label{eq: cech2}
\check{C}(\fU, \Obq_{\sE,\sL})\to \Obq_{\sE,\sL}(U)
\end{equation}
respects this filtration, hence induces a map of spectral sequences. The induced map on the associated graded spaces (the $E^1$ page) is the map
\[
\check{C}(\fU,\Obcl_{\sE,\sL}[\hbar])\to \Obcl_{\sE,\sL}(U)[\hbar],
\]
which was shown to be a quasi-isomorphism in the proof of Theorem \ref{thm: obcl}. Hence the map in Equation \ref{eq: cech2} is a quasi-isomorphism.
\end{proof}

\section{The main theorems}
\label{sec: main theorem}

In this section, we state and prove a generalization of Theorem~\ref{thm: main} that applies to a general free bulk-boundary field theory $\sE$ with boundary condition $\sL$. Without loss of generality , we will assume that the underlying manifold is of the form $M=\Mbdy\times \RR_{\geq 0}$, so that $\partial M = M_\partial$.  Let $\pi: M\to \Mbdy$ denote projection onto the boundary. We will also assume that the space of fields is globally of the form $\sE\,\hotimes\, \Omega^\bullet_{\RRge}$, with the pairing $\ip$ of the form specified in Definition~\ref{dfn: tnbft}. 

\begin{rmk}
The assumption that $M=\Mbdy\times \RR_{\geq 0}$ is purely for convenience. Our methods construct factorization algebras on an arbitrary manifold with boundary, so long as one can find a tubular neighborhood of the boundary on which the fields decompose to be ``topological normal to the boundary.'' After all, factorization algebras are local-to-global in nature, so we can patch together a construction near the boundary with a construction far into the bulk.
\end{rmk}

Here is our generalization of Theorem~\ref{thm: main} at the classical level.

\begin{thm}
\label{thm: maingenlcl}
For a free bulk-boundary field theory $(\sE,\sL)$, we have the following identifications:
\begin{enumerate}
\item Let $\Obcl$ denote the factorization algebra on $\mathring M$ of classical observables for $\sE$, constructed using the techniques of Chapter 4 of \cite{CG1}. Then, there is an isomorphism
\[
\Obcl_{\sE,\sL}\big|_{\mathring M} \cong \Obcl_\sE.
\]
\item There is a quasi-isomorphism
\begin{align}
\label{eq: qicl}
\II^{\cl}: \Obcl_\sL&\to \pi_*\Obcl_{\sE,\sL}.
\end{align}
\end{enumerate}
\end{thm}

We will state now the quantum analogue of this theorem before turning to the proofs.

\begin{thm}
\label{thm: maingenlq}
For a free bulk-boundary field theory $(\sE,\sL)$, we have the following identifications:
\begin{enumerate}
\item Let $\Obq$ denote the factorization algebra on $\mathring M$ of quantum observables for $\sE$, constructed using the techniques of Chapter 4 of \cite{CG1}. Then, there is an isomorphism
\[
\Obq_{\sE,\sL}\big|_{\mathring M} \cong \Obq_\sE.
\]
\item There is a quasi-isomorphism
\begin{align}
\label{eq: qiq}
\II^{\q}: \Obq_\sL&\to \pi_* \Obq_{\sE,\sL}.
\end{align}
\end{enumerate}
\end{thm}

\begin{rmk}
One consequence of this theorem is that the quantum boundary observables, for any choice of splitting $L^\perp$, are explicitly identified with $\pi_* \Obq_{\sE,\sL}$.
Hence we see again that the choice of splitting is irrelevant.
\end{rmk}

\begin{rmk}
Theorems \ref{thm: maingenlcl} and \ref{thm: maingenlq} are characterizations of the ``boundary value'' and the ``bulk value'' of the factorization algebras $\Obcl_{\sE,\sL}$, $\Obq_{\sE,\sL}$. However, the bulk-boundary factorization algebras contain more information than their bulk and boundary values alone---they also encode an action of the bulk observables on the boundary observables. This is a rich structure. For example, in the Poisson sigma model we believe the structure to be related to the formality quasi-isomorphism of Kontsevich \cite{KontPSM}. We study this action for topological mechanics and the Chern-Simons/chiral WZW system in Proposition \ref{prp: toplmech} and Lemma \ref{lem: cswzwonhalfline}, respectively.
\end{rmk}

We now turn to proving these theorems. 

\begin{proof}[Proof of classical theorem]
The first statement of the theorem follows immediately from the fact that $\condfields(U)=\sE(U)$ when $U\cap \del M=\emptyset$. 

It remains, therefore, to prove the second statement. 
Throughout the proof, let $U$ be an open subset of $\Mbdy$. 
Let us first construct the cochain map
\[
\II^{\cl}(U): \Obcl_\sL(U)\to \Obcl_{\sE,\sL}(U\times \RRge)
\]
for each open subset $U\subset M_\del$. 
To this end, let $\phi$ be a compactly-supported function on $\RRge$ whose integral over $\RRge$ is 1, and let $\Phi(t):=\int_0^t \phi(s)ds$. 
Both the boundary and bulk observables arise as symmetric algebras built on cochain complexes, 
so the map $\II^{\cl}$ will be induced from a cochain map on the linear observables. 

As a first step, we decompose the fields $\condfieldscs$ further. 
By hypothesis, we have the isomorphism
\[
\sE_{c}(U\times \RRge)\cong (\sEb)_{c}(U) \,\hotimes\, \Omega^\bullet_{\RRge,c}(\RRge).
\]
Recall that in the construction of the boundary observables, we have a decomposition
\[
\Qb=Q_L+Q_{L^\perp}+Q_{rel},
\] 
where $Q_L$ preserves $\sL$, $Q_{L^\perp}$ preserves $\sL^\perp$, and $Q_{rel}$ maps $\sL^\perp$ to $\sL$. We can therefore write
\[
\condfieldscs(U\times \RRge)\cong \left( \sL_{c}(U)\,\hotimes\, \Omega^\bullet_{\RRge,c}(\RRge)\rtimes \sL^\perp_{c}(U)\,\hotimes\, \dirforms(\RRge)\right),
\]
where $\dirforms(\RRge)$ is the cochain complex (concentrated in degrees 0 and 1)
\[
\begin{tikzcd}
\left\{f\in \Omega^0_{\RRge,c}(\RRge)\mid f(t=0)=0\right\}\ar[r,"d_{dR}"]&\Omega^1_{\RRge,c}(\RRge),
\end{tikzcd}
\]
and the symbol $\rtimes$ reminds us that $\sL^\perp$ is not a subcomplex of $\sEb$. 
Note that our boundary condition requires that only the $\sL^\perp$-valued fields vanish at the boundary. 

Define the map $\II^{\cl}(U): \sL^\perp_c(U)[-1]\to \condfieldscs(U\times \RR_{\geq 0})$ by 
\[
\II^{\cl}(U)(\alpha)=\alpha \wedge \phi \, \d t-(-1)^{|\alpha|}(\Phi-1)Q_{rel}\alpha,
\]
where $|\alpha|$ denotes the cohomological degree of $\alpha$ in $\sL^\perp$ (not $\sL^\perp[-1]$). 
The map $\II^{\cl}(U)$ is of cohomological degree zero because of the terms $\wedge \phi \,\d t$ and $Q_{rel}$. 
Moreover, $\II^{\cl}(U)(\alpha)$ does indeed have compact support if $\alpha$ does, since $(\Phi-1)(t)=0$ for $t>>0$. 
By construction $\II^{\cl}$ is a map of precosheaves. 
We also see that $\II^{\cl}(U)(\alpha)$ satisfies the boundary condition because 
\[
\rho\left( \II^{\cl}(U)(\alpha)\right)= (-1)^{|\alpha|}Q_{rel}\alpha
\]
and $Q_{rel}\alpha$ lives in $\sL_c$.
Finally, we check that $\II^{\cl}(U)$ is a cochain map:
on the one hand,
\[
\II^{\cl}(U)(Q_{L^\perp}\alpha) = (Q_{L^\perp}\alpha)\wedge \phi \, \d t-(-1)^{|\alpha|+1}(\Phi-1)Q_{rel}Q_{L^\perp}\alpha,
\]
and on the other,
\begin{align*}
Q\II^{\cl}(U)(\alpha)& = Q_{L^\perp}\alpha \wedge \phi \, \d t+ Q_{rel}\alpha \wedge \phi \, \d t\\
&-Q_{rel}\alpha \wedge \phi \, \d t -(-1)^{|\alpha|}(\Phi-1)Q_{L}Q_{rel}\alpha.
\end{align*}
Once one uses the relation $Q_L Q_{rel}=-Q_{rel}Q_{L^\perp}$, 
one sees that the two expressions are equal. 
Since $\II^{\cl}(U)$ respects the differentials on the complexes as well as the extension maps, 
it extends to a map of factorization algebras $\Obcl_\sL\to \pi_*\Obcl_{\sE,\sL}$.

It remains to show that $\II^{\cl}(U)$ is a quasi-isomorphism. 
We will exhibit, in fact, something much stronger: a deformation retraction.
Namely, we will produce a cochain map $\PP^{\cl}(U)$ such that $\PP^{\cl}(U)\II^{\cl}(U)=\id$ and 
a cochain homotopy $\KK^{\cl}(U)$ between $\II^{\cl}(U)\PP^{\cl}(U)$ and the identity~$\id$. 

To this end, consider the map
\[
\PP^{\cl}(U): \condfieldscs(U\times \RRge) \to \sL^\perp_c(U)[-1]
\]
where
\[
\PP^{\cl}(U)(e) =p_{L^\perp}\left( \int_{\RRge} e \right)
\]
and where $p_{L^\perp} \colon \sEb\to \sL^\perp$ is the canonical map induced by the quotient bundle map $\Eb \to L^\perp$.
Notice that 
\begin{align*}
\PP^{\cl}(U)\left(Q_{L^\perp}e+Q_{L}e+Q_{rel}e+(-1)^{|e|}\frac{de}{dt}\wedge \d t \right)
&=Q_{L^\perp} p_{L^\perp} \int_{\RRge}e+(-1)^{|e|}p_{L^\perp}\int_{\RRge}\frac{de}{dt}\wedge \dt
\\
&=Q_{L^\perp} p_{L^\perp} \int_{\RRge}e\\
&=Q_{L^\perp}\PP^{\cl}(U)(e),
\end{align*}
where the second equality holds because $e$ is compactly supported and $p_{L^\perp} e(0)=0$. 
Hence it is a cochain map.
Direct computation verifies that $\PP^{\cl}(U)\II^{\cl}(U)=\id$.

Consider now the degree --1 map
\[
\KK^{\cl}(U): \condfieldscs(U\times \RRge)\to \condfieldscs(U\times \RRge)
\]
where
\[
\big(\KK^{\cl}(U)(e)\big)(t)= (-1)^{|e|-1}(\Phi(t)-1)(\PP^{\cl}(U)(e))-(-1)^{|e|}\int_t^\infty e(s).
\]
The field $\KK^{\cl}(U)(e)$ satisfies the required boundary condition because 
\[
\KK^{\cl}(U)(e)(0)=(-1)^{|e|}p_{L^\perp}\int_{\RRge}e-(-1)^{|e|}\int_{\RRge}e
\]
and hence $\KK^{\cl}(U)(e)(0)$ is an element of $\sL$.
Direct computation shows that $\KK^{\cl}(U)$ is a cochain homotopy between $\II^{\cl}(U)\PP^{\cl}(U)$ and the identity. 

Just as $\II^{\cl}(U)$ extends to a map of symmetric algebras, extend $\KK^{\cl}(U)$ and $\PP^{\cl}(U)$ to maps
\begin{align*}
\KK^{\cl}(U)&:\Sym(\condfieldscs[1](U\times \RRge))=\left(\pi_*\Obq_{\sE,\sL}\right)(U)\to \left(\pi_*\Obq_{\sE,\sL}\right)(U)\\
\PP^{\cl}(U)&: \Obq_{\sE,\sL}(U)\to \Sym(\sL^\perp_c(U)) = \Obcl_\sL(U)
\end{align*}
by the usual procedure extending a deformation retraction at the linear level to symmetric powers. 
(One treatment with the necessary formulas is Section 2.5 of~\cite{othesis}.)
\end{proof}

Proving the quantum theorem is a modest modification of the classical argument.

\begin{proof}[Proof of quantum theorem]
The first statement of the theorem again follows immediately from the fact that $\condfields(U)=\sE(U)$ when $U\cap \del M=\emptyset$. 

It remains, therefore, to prove the second statement,
using the constructions from the proof of the classical theorem.
Throughout the proof, let $U$ be an open subset of $\Mbdy$. 
Recall that the cocycle $\mu$ determines $\Obq_\sL$ and the cocycle $\ip$ determines $\Obq_{\sE,\sL}$. 
We will show that $\II^{\cl}(U)$ respects the cocycles and hence determines the desired map $\II^q$ between the quantized factorization algebra.
In particular, we must show that
\[
\mu(\alpha_1,\alpha_2)=\ip[\II^{\cl}\alpha_1,\II^{\cl}\alpha_2].
\]
To see this, compute 
\begin{align*}
\ip[\II^{\cl}\alpha_1,\II^{\cl}\alpha_2]&= - \int_{\Mbdy\times \RRge} \ip[\alpha_1,Q_{rel}
\alpha_2]_{loc, \del}\phi (1-\Phi) \, \d t - (-1)^{|\alpha_1|}\int_{\Mbdy\times \RRge} \ip[Q_{rel}\alpha_1,
\alpha_2]_{loc, \del}\phi (\Phi-1) \, \d t\\
&= 2\mu(\alpha_1,\alpha_2)\int_{\RRge}\phi(1-\Phi)\, \d t.
\end{align*}
Since $\frac{d}{dt}(\Phi-1)=\phi$, we find
\[
\int_{\RRge}\phi(1-\Phi)= \int_{u=0}^1 u \, \d u = \frac{1}{2},
\]
which verifies that $\II^{\cl}$ is a cochain map at the level of quantum observables, as needed.
\end{proof}

\section{Applications}
\label{sec: apps}

In this section, we apply our main theorems to several bulk-boundary field theories,
namely the examples already mentioned in Section~\ref{sec: examples}.
In the low-dimensional examples, we can relate the factorization algebras to more familiar objects, 
such as associative algebras and vertex algebras.
For instance, in the example of topological mechanics, we find that our procedure is equivalent to the canonical quantization of the algebra $\sO(V)$ (on the ``bulk'' line) and its Fock space (on the boundary point).

In higher dimensions, our factorization algebras recover familiar phenomena when using simple product spaces and performing compactifications ({\em aka} pushforwards). For example, on a slab of the form $N\times [0,1]$ with $N$ an oriented 2-manifold, the CS/WZW system is shown to be equivalent to the free massless scalar boson on $N$ (see Lemmas \ref{lem: fullwzwscalarcl} and~\ref{lem: fullwzwscalarq}).

One payoff of our approach is that we get nontrivial constructions when we use interesting manifolds with boundary,
thanks to factorization homology.
For instance, in the example of the Poisson sigma model, we find that the global observables depend in an interesting way on the genus and the number of boundary components of the surface.
Similarly, Koszul duality makes an appearance by pointing transverse boundary conditions on either side of a strip (cf. \cite{Shoikhet, calaque_felder_ferrario_rossi_2011}), 
although we do not go far in this direction here.

\subsection{Topological mechanics}
\label{subsec: toplmech}

In this subsection, we study the factorization algebras for topological mechanics with values in $V$ and with boundary condition $L$. We will see that the factorization algebra of classical bulk-boundary observables encodes the commutative algebra $\Sym(V)$ together with the module $\Sym(V/L)$. For the quantum observables, we obtain the Weyl algebra $W(V)$ and the Fock module $F(L)$ built on $L$. (We define these objects in the sequel.)

Recall that a  symplectic vector space $(V,\omega)$ together with a Lagrangian subspace $L\subset V$ define a free bulk-boundary system on $[0,\epsilon)$, which we call topological mechanics (cf. Example~\ref{ex: toplmech}). 
(We can take $V$ to be $\ZZ$-graded, if we like, but of bounded total dimension.)
The main theorem \ref{thm: maingenlcl} identifies $\left.\Obcl_{\sE,\sL}\right|_{(0,\epsilon)}$ with the factorization envelope on $(0,\epsilon)$ of the abelian Lie algebra $V$. 
Proposition 3.4.1 of \cite{CG1} shows that this factorization algebra is equivalent to the locally constant factorization algebra on $(0,\epsilon)$ corresponding to the associative algebra $\sO(V):=\Sym(V^\vee)$. 
Similarly, $\left.\Obq_{\sE,\sL}\right|_{(0,\epsilon)}$ is equivalent to the factorization algebra on $(0,\epsilon)$ corresponding to the Weyl algebra $W(V)$.
(Recall that the Weyl algebra is the algebra generated by $V$ and $\hbar$ and subject to the relation $v_1v_2-v_2v_1=\omega(v_1,v_2)\hbar$.) 

The main theorems also identify the bulk-boundary observables $\Obcl_{\sE,\sL}([0,\delta))$ and $\Obq_{\sE,\sL}([0,\delta))$ with the boundary observables
\[
\Obcl_{\sL}\cong \Sym(V/L)
\]
and
\[
\Obq_{\sL}\cong \Sym(V/L)[\hbar],
\]
respectively, for any $\delta\leq\epsilon$. The second isomorphism arises from the fact that $Q_{rel}=0$. 

These identifications are purely identifications of factorization algebras on $\{0\}$; they do not take into account the actions of $\Obcl_\sE$ and $\Obq_\sE$ on the boundary observables. In this subsection, we show how the bulk and boundary observables interact through the bulk-boundary factorization algebras $\Obcl_{\sE,\sL}$ and $\Obq_{\sE,\sL}$. 
Namely, we will examine the structure maps involving one or more intervals including the boundary point.
These structure maps will give the boundary observables the structure of a right module over the corresponding algebras in the bulk.

More precisely, given an algebra $A$ and a pointed right module $M$ of $A$, there is a stratified locally constant factorization algebra $\cF_{A,M}$ on $[0,\epsilon)$ which assigns $A$ to any open interval, and $M$ to any half-closed interval (cf. \S 3.3.1 of \cite{CG1}). We will show that the cohomology factorization algebras $H^\bullet\Obcl_{\sE,\sL}$ and $H^\bullet \Obq_{\sE,\sL}$ will be of this form for particular choices of $A$ and $M$. We have already discussed that the corresponding algebras are $\sO(V)$ and $W(V)$ for the classical and quantum observables, respectively. It remains only to identify the relevant modules.

The Lagrangian $L\subset V$ determines a (right) module for the commutative algebra $\sO(V)$, namely $\sO(V/L)$ with the module structure induced from the restriction map. 
Similarly, $L$ determines a right module $F(L)$ for the Weyl algebra, namely the quotient of $W(V)$ by the right-submodule generated by $L$. The underlying vector space for $F(L)$ is $\Sym(V/L\oplus \hbar)$. Having established all the relevant notation, we can now state the main proposition.

\begin{prp}
\label{prp: toplmech}
For $(\sE,\sL)$ corresponding to topological mechanics of Example~\ref{ex: toplmech}, 
there is a quasi-isomorphism of factorization algebras on~$\RR_{\geq 0}$
\[
\Obcl_{\sE,\sL} \xto{\simeq} \cF_{\sO(V),\sO(V/L)}
\]
from classical observables to the stratified locally constant one associated to functions $\sO(V)$ and the module $\sO(V/L)$.
Likewise, there is a quasi-isomorphism of factorization algebras
\[
\Obq_{\sE,\sL} \xto{\simeq} \cF_{W(V),F(L)}
\]
from the quantum observables and the stratified locally constant one associated to the Weyl algebra $W(V)$ and the Fock module~$F(L)$.
\end{prp}

Note that this proposition says that the cohomology of the observables is wholly concentrated in degree zero,
so that the observables are determined precisely by the usual information in mechanics.

\begin{proof}
It is straightforward to verify that the factorization algebras of both the classical and quantum bulk-boundary observables are stratified locally constant with respect to the stratification $\{0\}\subset [0,\epsilon)$. 
Hence, each factorization algebra corresponds to some pair $(A,M)$. 
We need only to determine the modules living on the boundary. 
To this end, let $I_1=(0,\epsilon)$ and $I_2=[0,\epsilon)$. 
Consider the structure maps for the inclusion $I_1\subset I_2$. Let $A$ stand momentarily for either of $\sO(V)$, $W(V)$, and similarly let $M$ stand for either of the two modules on the boundary. The structure map $m_{I_1}^{I_2}$ induces a map $A\to M$. The associativity axiom of a prefactorization algebra guarantees that this is a map of $A$ modules.

Recall from the proof of Theorem~\ref{thm: maingenlcl} that the map $\II^{\cl}$ is induced from a choice $\phi$ of compactly-supported function on $I_2$ whose total integral is 1. Let us suppose that $\phi$ is supported on $I_1$. Then, we have a quasi-isomorphism
\[
\II^{\cl}_{int}: V\to \condfieldscs(I_1)[1]
\]
where
\[
\II^{\cl}_{int}(v) = \phi \, \d t\otimes v.
\]
The symmetrization of this map, which we also denote by $\II^{\cl}_{int}$, induces a quasi-isomorphism 
\[
\II^{\cl}_{int}:\Sym(V) \to \Obcl_{\sE,\sL}(I_1).
\]
Consider the composite map
\[
\xymatrix{
\Sym(V)\ar[r]^-{\II^{\cl}_{int}}& \Obcl_{\sE,\sL}(I_1)\ar[r]^-{m_{I_1}^{I_2}}&\Obcl_{\sE,\sL}(I_2)\ar[r]^-{\PP^{\cl}}&\Sym(V/L),
}
\]
where $\PP^{cl}(I_2)$ is introduced in the proof of Theorem \ref{thm: maingenlcl}. It follows directly from the definitions that the composite is the map $\Sym(V)\to \Sym(V/L)$ induced from the projection $V\to V/L$. The statement of the proposition for the classical observables follows.

We now ``perturb'' the classical information. 
We would like to understand the structure map 
\[
m_{I_1}^{I_2}:\Obq_{\sE,\sL}(I_1)\to \Obq_{\sE,\sL}(I_2)
\] 
at the level of cohomology. 
We know that the cohomology of $\Obq_{\sE,\sL}(I_1)$ is the underlying vector space of $W(V)$, and the cohomology of $\Obq_{\sE,\sL}(I_2)$ is $\Sym(V/L)[\hbar]$, which is  the underlying vector space of a module $M$ for $W(V)$. 
The structure map $m_{I_1}^{I_2}$ induces a map $T: W(V)\to M$ which intertwines the right $W(V)$ actions. 
Because $\Obq_{\sE,\sL}$ is filtered by powers of $\hbar$, and because the associated graded factorization algebra is $\Obcl_{\sE,\sL}\otimes_{\CC}\CC[\hbar]$, $T$ is surjective. 
Hence, to understand $M$, we simply need to identify the kernel of $T$. In the proof of Theorem \ref{thm: maingenlcl}, we constructed maps $\II^{cl}(I_2),\PP^{cl}(I_2),\KK^{cl}(I_2)$ which fit into a deformation retraction. 
Hence, the homological perturbation lemma (see, e.g., \cite{crainic}) gives a formula for a quasi-isomorphism 
\[
\PP^{\q}: \Obq_{\sE,\sL}(I_2)\to \Sym(V/L)[\hbar].
\] 
On the sub-complex $\sE_{\sL,c}(I_2)[1]\subset \Obq_{\sE,\sL}(I_2)$, $\PP^q$ agrees with $\PP^{\cl}$. 
Moreover, as demonstrated in \cite{CG1}, the map 
\[
\xymatrix{
V\ar[r]^-{\II^{\cl}_{int}} &\sE_{\sL,c}(I_1)[1]\ar@{^{(}->}[r] &\Obq_{\sE,\sL}(I_1)
}
\]
induces the canonical map $V\to W(V)$ on cohomology. 
Finally, tracing through the definitions, the composite 
\[
\xymatrix{
V\ar[r]^-{\II^{\cl}_{int}} &\sE_{\sL,c}(I_1)[1]\ar@{^{(}->}[r] &\Obq_{\sE,\sL}(I_1)\ar[r]^-{m_{I_1}^{I_2}}&\Obq_{\sE,\sL}(I_2)\ar[r]^-{\PP^{\q}}&M
}
\] 
is seen to be the quotient map $V\to V/L$ followed by the inclusion $V/L\to \Sym(V/L)[\hbar]$. Thus, $L$ is in the kernel of $T$, and hence so too is the whole submodule of $W(V)$ generated by $L$. For dimension reasons, this implies that $M\cong F(L)$.  
\end{proof}

\subsection{Free Poisson sigma model}

We examine Example~\ref{ex: psm} and explore some interesting consequences.
For instance, we show that our results recover naturally --- for this simple class of Poisson spaces and coisotropic submanifolds ---  the well-known Swiss cheese algebras.
We then turn to a discussion of Koszul duality, 
and finally to what happens with higher genus surfaces.

\subsubsection{The boundary observables}

Recall from Example~\ref{ex: psm} that a vector space $V$ and a skew-symmetric linear map $\Pi: V^\vee\to V$ determine a field theory on any oriented surface $\Sigma$ with boundary. The (underlying graded vector) space of fields of this theory is 
\[
\Omega_\Sigma^\bullet \otimes V^\vee [1]\oplus \Omega_\Sigma^\bullet\otimes V.
\] 
The differential is of the form $(\d_{dR}\otimes \id)\oplus (\id \otimes \Pi)$ (the first term preserves the decomposition above, and the second term maps the first summand to the second). In particular, one obtains a field theory on the upper half-plane $\Sigma = \HH$. For the choice of Lagrangian $\sL_0= \Omega^\bullet_{\RR}\otimes V$, we have $\sL_0^\perp= \Omega^\bullet_\RR\otimes V^\vee[1]$. 
The following lemma is a straightforward application of Proposition 3.4.1 of~\cite{CG1}.

Recall our notational convention: given an associative algebra $A$, let $\cF_A$ denote the locally constant factorization algebra on $\RR$ constructed from $A$ (cf. the first example in Section 3.1.1 of~\cite{CG1}).

\begin{lmm}
\label{lmm:bdryL0}
For $\sL_0$ as defined just above, 
there is a quasi-isomorphism of factorization algebras on~$\RR$
\[
\Obcl_{\sL_0} \xto{\simeq} \cF_{\Sym(V^\vee)}
\]
for classical observables and a quasi-isomorphism of factorization algebras
\[
\Obq_{\sL_0} \xto{\simeq} \cF_{(\Sym(V^\vee)[\hbar],\ast)},
\]
where $\ast$ refers to the Kontsevich star product on $\Sym(V^\vee)[\hbar]$. 
The product is characterized by the relation
\[
\nu_1\ast\nu_2-\nu_2\ast\nu_1 = \hbar \nu_1(\Pi\nu_2),
\]
where $\nu_1,\nu_2\in V^\vee$. 
\end{lmm}

In our situation the dual space $V^\vee$ can be decomposed into a direct sum of a vector space $V_t^\vee$ with a trivial pairing and a vector space $V_s^\vee$ with a nondegenerate (i.e., symplectic) pairing.
Thus the quantum observables corresponds to a tensor product of a commutative algebra $\Sym(V_t^\vee)$ and a Weyl algebra~$W(V_s)$.
It is thus natural to analyze just these two cases since the general answer can be assembled from them.

\begin{rmk}
One can define the boundary observables $\Obq_{\sL_0}$ without making any reference to the theory on $\HH$ for which $\sL_0$ is a boundary condition. 
Theorem~\ref{thm: maingenlq} tells us that these observables are equivalent to the pushforward of the coupled bulk-boundary observables. 
When $V$ ceases to be a linear Poisson manifold, there is no direct definition of $\Obq_{\sL_0}$, and in fact the quantization of $\Sym(V^\vee)$ produced in \cite{KontPSM} requires in an essential way the study of the theory on $\HH$. 
We expect that once one constructs the factorization algebra of quantum observables for the (interacting) Poisson sigma model on $\HH$, its pushforward to $\RR$ recovers the algebra $\Sym(V^\vee)[\hbar]$ with the Kontsevich star product. 
In fact, we expect that one can apply a similar procedure to any theory with a shifted Poisson structure, using the so-called ``universal bulk theory''~\cite{ButsonYoo}.
\end{rmk}

As a special case, when $\Pi$ is zero, the quantum observables correspond to the commutative algebra generated by $V^\vee$ and $\hbar$; 
moreover, a new boundary condition becomes available. 
Namely, we take $\sL_1= \Omega^\bullet_\RR \otimes V^\vee[1]$, so that $\sL^\perp_1 = \Omega^\bullet_\RR\otimes V$. We have a similar lemma.

\begin{lmm}
There are quasi-isomorphisms 
\begin{align*}
\Obcl_{\sL_1}&\xto{\simeq} \cF_{\Lambda^\bullet V}\\
\Obq_{\sL_1}& \xto{\simeq} \cF_{(\Lambda^\bullet V)[\hbar]};
\end{align*}
where we use $\Lambda^\bu V$ to denote the free graded algebra generated by $V$ in degree~$1$.
\end{lmm}

\subsubsection{The bulk algebra and Swiss cheese-type structures}

On half-space $\HH$ our bulk-boundary factorization algebra is locally constant with respect the stratification $\partial \HH \subset \HH$.
By \cite{higheralgebra, AyaFraTan} it thus encodes an algebra over the Swiss cheese operad.
It is interesting to ask whether we recover the expected Swiss cheese algebra,
and it is easy to see that we do for the classical observables.
Let $\cO_\Pi$ denote the dg commutative algebra consisting of the polyvector fields $\Sym(V^\vee \oplus V[-1])$ equipped with the differential~$[\Pi, -]$.

\begin{lmm}
There is a natural quasi-isomorphism
\[
\Obcl_\sE \xto{\simeq} \cF_{\cO_\Pi}
\] 
of factorization algebras on $\mathring \HH$ where $\cF_{\cO_\Pi}$ denotes the locally constant factorization algebra on $\mathring \HH$ associated to the dg commutative algebra~$\cO_\Pi$.
\end{lmm}

In other words, the classical observables encode precisely the desired dg commutative algebra.

\begin{proof}
We start by verifying the claim when $\Pi = 0$ and then explain how to deform to the general case.
Consider the canonical inclusion of the constant sheaf $\CC$ into the de Rham complex $\Omega^\bullet_{\mathring \HH}$,
which induces a canonical quasi-isomorphism of sheaves
\[
i: V \oplus V^\vee[1] \hookrightarrow \Omega^\bullet_{\mathring\HH} \otimes (V \oplus V^\vee[1]).
\]
This map determines a map on cosheaves of dg commutative algebras
\[
i^*: \Obcl_{\Pi = 0} \to \cO_{\Pi = 0},
\]
and hence the desired quasi-isomorphism.
Turning on $\Pi$, we deform the differential on the polyvector fields,
which also deforms the differential on the classical observables,
but this deformation is manifestly $\Omega^\bullet$-linear, 
as it only depends on the target space and not on the source.
\end{proof}

This quasi-isomorphism is compatible with the quasi-isomorphism of Lemma~\ref{lmm:bdryL0}:
the quotient map $\cO_\Pi \to \Sym(V^\vee)$ fits in a commuting square with the structure map $\Obcl_{\Pi = 0}(\RR^2) \to \Obcl_{\sL_0}(\RR^1)$ determined by a disk in the bulk sitting inside a semi-disk intersecting the boundary.

We now turn to the quantum case.
Here a comparison is a bit more hairy, 
because it is more complicated to compare an $E_2$ algebra to a $P_2$ (or Gerstenhaber) algebra.
Even worse, our construction produces explicitly a locally constant factorization algebra on $\RR^2$,
and it requires work to unpack its associated $E_2$ algebra.
As a step in this direction, we note that the underlying cochain complex of the quantum observables on a disk is quasi-isomorphic to~$\cO_\Pi[\hbar]$.
Indeed, the quasi-isomorphism for the classical observables carries over immediately:
it remains a quasi-isomorphism at the quantum level because the BV Laplacian is nontrivial only on elements that are annihilated by the map.

\begin{lmm}
The natural quasi-isomorphism $\Obcl_\sE(\RR^2) \xto{\simeq} \cO_{\Pi}$ quantizes to a quasi-isomorphism $\Obq_\sE(\RR^2) \xto{\simeq} \cO_{\Pi}[\hbar]$ of cochain complexes. 
\end{lmm}

We do not verify here that the expected Gerstenhaber structure appears under this identification,
as it would require developing machinery orthogonal to the efforts of this paper.
We remark, however, that so long as $\Pi \neq 0$, 
the quantum observables determine a nontrivial $E_2$ deformation of $\Obcl_\sE$.
This fact can be seen by examining the dimensional reduction of the linear Poisson sigma model from $\RR \times S^1$ to $\RR$;
the quantum observables there have a nontrivial commutator.

\subsubsection{A case of Koszul duality}

By combining the boundary conditions in the case $\Pi = 0$, 
we get an appealing view on the archetypal Koszul duality between the algebras $\Sym(V^\vee)$ and $\Lambda^\bu V$.
Consider the Poisson sigma model on the slab $\RR\times [0,1]$ with the boundary conditions $\sL_0$ at $t=0$ and $\sL_1$ at $t=1$,
and let $p: \RR\times [0,1] \to \RR$ denote the canonical projection,
as in Figure~\ref{fig:koszulstrip}. 
We will denote this total boundary condition by~$\sL$. 

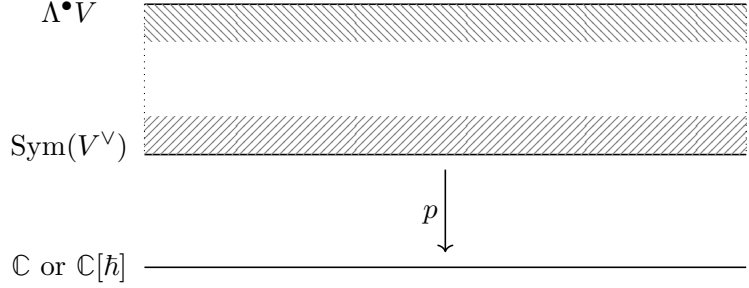
\begin{figure}

\begin{tikzpicture}
\draw[semithick](-4,1) -- (4,1);
\fill[pattern=north west lines, pattern color=gray] (-4,1) rectangle (4,0.5);
\draw[semithick](-4,-1) -- (4,-1);
\fill[pattern=north east lines, pattern color=gray] (-4,-.5) rectangle (4,-1);
\draw[dotted](-4,1) -- (-4,-1);
\draw[dotted](4,1) -- (4,-1);
\draw[->,semithick] (0,-1.2) -- (0,-2.3);
\draw[semithick](-4,-2.5) -- (4,-2.5);
\node at (-.2,-1.8){$p$};
\node at (-5,-.9){$\Sym(V^\vee)$};
\node at (-5,.9){$\Lambda^\bu V$};
\node at (-5,-2.5){$\CC$ or $\CC[\hbar]$};
\end{tikzpicture}
\caption{Strip equipped with transverse boundary conditions}
\label{fig:koszulstrip}
\end{figure}

Something interesting happens here: the composite system looks trivial,
so that the observables are the unit factorization algebra. 

\begin{lmm}
\label{lmm: pairing}
For the Poisson sigma model with $\Pi=0$ on $\RR\times [0,1]$ with boundary condition $\sL$, 
we have an equivalence of factorization algebras 
\[
p_* \Obcl_{\sE,\sL}\simeq \CC
\]
on the real line,
where $\CC$ is the factorization algebra which assigns $\CC$ to all open subsets~$U\subset \RR$.
Likewise,
\[
p_* \Obq_{\sE,\sL}\simeq \CC[\hbar],
\]
where $\CC[\hbar]$ also denotes the locally constant factorization algebra.
\end{lmm}

The essential reason for this result is that the Lagrangian boundary conditions are transverse,
so that the only solution is trivial.
A careful proof is below.

But the geometry here encodes interesting relations on the algebraic structures.
\begin{itemize}
\item The inclusion $\RR \times [0,t) \hookrightarrow \RR\times [0,1]$ leads to a structure map of factorization algebras 
from the boundary classical observables to the global observables;
in terms of algebras, it corresponds to the augmentation $\Sym(V^\vee) \to \CC$ corresponding to the quotient by the ideal~$(V^\vee)$.
This augmentation quantizes to an augmentation $\Sym(V^\vee)[\hbar] \to \CC[\hbar]$.
\item The inclusion on the other boundary component encodes the standard augmentation $\Lambda^\bullet V \to \CC$, for the classical observables, and the analog over $\CC[\hbar]$ for the quantum observables. 
\item Consider now the inclusion 
\[
\RR \times \left( [0,t) \sqcup (1-t,1] \right) \hookrightarrow \RR\times [0,1].
\]
The structure map of factorization algebras here encodes a nontrivial pairing of algebras
\[
q: \Sym(V^\vee) \otimes \Lambda^\bu V \to \CC 
\]
whose restriction to each factor is the standard augmentation.
The same holds in the quantum case.
\end{itemize}
All these maps are essentially induced by the inclusion of the empty open set,
which gives a distinguished line inside the observables on each open.

The pairing $q$ witnesses the algebras as Koszul dual,
because the (derived) hom-tensor adjunction yields a canonical map
\[
\Lambda^\bu V \to \RR \End_{\Sym(V^\vee)}(\CC, \CC) =: \Sym(V^\vee)^{!}
\]
and it is a quasi-isomorphism here.
(See Section 5.2.5 of~\cite{higheralgebra} for a careful treatment, or~\cite{LurieLecture} for a beautiful exposition,
of Koszul duality along these lines.)  

\begin{rmk}
We find this approach complementary to prior work connecting Koszul duality to deformation quantization \cite{Shoikhet, calaque_felder_ferrario_rossi_2011}.
It also provides a concrete, physical example of the general, abstract approach to Koszul duality via factorization algebras \cite{AyaFraPK, Matsuoka2015}.
For a further and deeper discussion of how Koszul duality fits into physics, particularly topological field theories,
see Appendix A of \cite{CP}.
For a connection with holography, see~\cite{CostelloLiSUGRA}.
\end{rmk}

\begin{proof}[Proof of Lemma~\ref{lmm: pairing}]
Given an open subset $U\subset \RR$, we construct a contracting homotopy for 
\[
\condfieldscs(U\times [0,1])[1].
\]
The space of fields is 
\[
\sE=\Omega^\bullet_{\RR\times[0,1]}\otimes V^\vee[1]\oplus \Omega^\bullet_{\RR\times[0,1]}\otimes V.
\]
Let us write a general field in the form $A+B$, where $A$ and $B$ lie in the first and second summands of $\sE$. 
Let $t$ denote the coordinate on $[0,1]$, and let $\iota_0$ and $\iota_1$ denote the inclusions $\RR\subset \RR\times[0,1]$ at the $t=0$ and $t=1$ coordinates, respectively. 
The boundary condition $\sL$ requires $\iota_0^* A=0$, $\iota_1^* B=0$. 
Consider fields of the type $A=\alpha\otimes \nu$, $B=\beta\otimes \mu$, where $\alpha\in \Omega^\bullet_{\RR}\otimes V^\vee[1]$, $\beta\in \Omega^\bullet_{\RR}\otimes V$, $\nu,\mu\in \Omega^\bullet_{[0,1]}$.
Let $\eta_0$ denote the degree --1 endomorphism of the de Rham forms $\Omega^\bullet_{[0,1]}$ which takes a one-form $\nu$ to the unique anti-derivative of $\nu$ which vanishes at $t=0$. 
Similarly, define $\eta_1$ to be the anti-derivative which vanishes at $t=1$.  For each open subset $U\subset \RR$, define the map
\[
K: \condfieldscs(U\times [0,1])\to \condfields(U\times [0,1])[-1]
\]
as follows:
\begin{align*}
K(\alpha\otimes \mu)&= (-1)^{|\alpha|}\alpha\otimes \eta_0(\mu)\\
K(\beta\otimes \nu)&= (-1)^{|\beta|}\beta\otimes \eta_1(\nu).
\end{align*}
With the domain and codomain written as above, $K$ is a degree 0 map; it can also be understood as a degree --1 map from $\condfieldscs$ to itself.
Here, we note that a general field $A+B$ cannot be written literally in the form $\alpha\otimes \mu+\beta\otimes \nu$, or even as a finite sum of such elements; however, the formulas for $K$ written above specify $K$ on the full (completed bornological) tensor product uniquely. Another way to interpret the above formulas is that $K$ acts as $1\otimes \eta_0$ on the $A$ fields and $1\otimes \eta_1$ on the $B$ fields. 
One verifies readily that $d_{dR}K+Kd_{dR}=\id$, so that $\condfieldscs(U\times[0,1])$ is acyclic for all $U$. The lemma follows.
\end{proof}

\begin{rmk}
We expect that the analogous arguments for 2-dimensional BF theory give the Koszul duality between, on the one hand, the factorization algebra corresponding to~$U\fg$ and, on the other hand, the factorization algebra corresponding to~$C^\bullet (\fg)$.
\end{rmk}

\subsubsection{The structure of the global sections}

We now consider the Poisson sigma model on a compact, connected, oriented surface $M$ with non-empty boundary $\del M$. 
Let $M$ have genus $g$ and $b>0$ boundary components. 
Our main result is the following.

\begin{lmm}
\label{lem: psmconformalblocks}
For the Poisson sigma model with boundary condition $\sL_0$, 
the cohomology of the global quantum observables 
\[
H^\bullet \left(\Obq_{\sE,\sL_0}(M)\otimes_{\RR[\hbar]} \RR_{\hbar =1}\right)
\] 
has rank 1 over $\RR$ and is concentrated in degree~$-2g(\dim \ker \Pi)-b\dim V$.
\end{lmm}
 
\begin{rmk}
\label{rmk: states} 
In all observable complexes appearing in this remark, we tacitly set $\hbar = 1$ but omit this from the notation.
(More accurately, we perform the tensor product $- \otimes_{\RR[\hbar]}\RR_{\hbar=1}$ for all observable complexes.) 
We begin by making two claims whose proofs are straightforward:
\begin{enumerate}
\item The cohomology $H^\bullet (\Obq_{\sL_0}(\del M))$ of the boundary observables is the $b$-fold tensor product of the Hochschild homology of the algebra $A:=(\Sym(V^\vee), \ast_{\hbar=1})$ with itself. 
\item The Hochschild homology of $A$ is concentrated in degrees 
\[
-\dim V, \cdots , -(\dim V - \dim \ker \Pi).
\]
\end{enumerate} 
Choose a tubular neighborhood $T$ of $\del M$. 
The structure map for the inclusion $T\subset M$ induces a map
\[
\tau: H^\bullet(\Obq_{\sL_0}(\del M))\cong H^\bullet(\Obq_{\sE,\sL_0}(T)) \to H^\bullet(\Obq_{\sE,\sL_0}(M)).
\]
Therefore, we find that the domain of $\tau$ is concentrated in negative degrees $\geq -b\dim V$,
while the codomain is concentrated in degree~$-2g(\dim \ker \Pi)-b\dim V$.
There are a number of cases to consider.
First, for $g\neq 0$ and $\dim \ker \Pi\neq 0$, the map $\tau$ is manifestly trivial, simply because its domain and codomain are concentrated in different cohomological degrees. 
However, when $g=0$ or $\dim\ker \Pi=0$, the map $\tau$ could be non-trivial, since its domain is concentrated in negative degrees $\geq -b\dim V$ and its codomain is concentrated in degree~$-b\dim V$.
Since the Hochschild homology of $A$ is a refinement of the quotient $A/[A,A]$, 
one should understand $\tau$ as defining a (possibly trivial) trace on $A$.
Therefore, $\tau$ is a sort of state of the quantum mechanical system on $\RR$ defined by~$A$.
We expect this trace to be non-zero exactly when $\Pi$ is non-degenerate.
\end{rmk}

\subsubsection{Proof of Lemma \ref{lem: psmconformalblocks}}

We prove Lemma~\ref{lem: psmconformalblocks} in steps.
First, recall the form of the cohomology and relative cohomology of surfaces with boundary. 

\begin{lmm}
The absolute and relative cohomology groups of $M$ and $(M,\del M)$ with coefficients in $\RR$ are
\[
H^\bullet(M) \cong \left \{
\begin{array}{lr}
\RR & \bullet =0\\
\RR^{2g+b-1} & \bullet =1 \\
0 & \bullet = 2\\
0 & else
\end{array}
\right.
\]
and
\[
H^\bullet(M,\del M) \cong \left \{
\begin{array}{lr}
0 & \bullet = 0\\
\RR^{2g+b-1} & \bullet =1\\
\RR & \bullet =2\\
0 & else
\end{array}
\right.,
\]
respectively. 
\end{lmm}

In our computation of the observables, we need to understand more than the isomorphism classes of the cohomology groups $H^\bullet(M)$ and $H^\bullet(M,\del M)$. 
We also need to understand two further pieces of information:  
the structure of the map 
\[
\delta: H^1(M,\del M) \to H^1(M)
\]
induced from the inclusion $\Omega^\bullet(M,\del M)\to \Omega^\bullet(M)$ and the structure of the Lefschetz duality pairing:
\[
\Omega: H^1(M,\del M) \otimes H^1(M) \to \RR.
\]
That is the aim of the following lemma which follows directly from Lemma 1.2 of \cite{confspacesmfldsbdy}.

\begin{lmm}
\label{lem: structureofcohom}
There are isomorphisms 
\[
H^1(M,\del M) \cong \RR^{2g}\oplus \RR^{b-1}
\]
\[
H^1(M) \cong \RR^{2g}\oplus \RR^{b-1}
\]
so that 
\begin{enumerate}
\item with respect to these decompositions, $\delta$ has the block-diagonal form
\[
\left( 
\begin{array}{cc}
T&0\\
0&0
\end{array}
\right),
\]
where $T: \RR^{2g}\to \RR^{2g}$ is an isomorphism, and
\item under these isomorphisms, $\Omega$ is the sum of a bilinear form on $\RR^{2g}$ and a bilinear form on~$\RR^{b-1}$.
\end{enumerate}
\end{lmm}

\begin{crl}
The space of bulk-boundary fields for the Poisson sigma model with boundary condition $\sL_0$ has the following cohomology:
\[
H^0(\sE_{\sL_0}(M)) \cong (\RR^{2g}\otimes \ker \Pi )\oplus(\RR^{b-1}\otimes V^\vee)\oplus V,
\]
\[
H^1(\sE_{\sL_0}(M))\cong(  \RR^{2g}\otimes \coker\, \Pi )\oplus(\RR^{b-1}\otimes V)\oplus V^\vee,
\]
where the summand $\RR^{b-1}\otimes V^\vee$ (respectively $\RR^{b-1}\otimes V$) arises from the $\RR^{b-1}$ summand in $H^1(M,\del M)$ (respectively $H^1(M)$) of the preceding Lemma. 
\end{crl}

\begin{proof}
The complex $\sE_{\sL_0}(M)$ is filtered, with $F^0(\sE_{\sL_0})=\Omega^\bullet(M) \otimes V$ and $F^1(\sE_{\sL_0}(M))=\sE_{\sL_0}(M)$. 
The $E^0$ page of the corresponding spectral sequence is $H^\bullet(M,\del M)\otimes V^\vee[1]\oplus H^\bullet(M)\otimes V$.
The differential on the $E^0$ page is 
\[
\delta\otimes \Pi : H^1(M,\del M)\otimes V^\vee\to H^1(M)\otimes V[-1].
\]
Using the splitting of Lemma \ref{lem: structureofcohom}, and choosing splittings $V\cong \coker\, \Pi \oplus V'$, $V^\vee \cong \ker \Pi \oplus (V')^\vee$,  we arrive at the conclusion of the Corollary.
\end{proof}

\begin{proof}[Proof of Lemma \ref{lem: psmconformalblocks}]
The complex $\Obq_{\sE,\sL_0}(M)\otimes_{\RR[\hbar]}\RR_{\hbar=1}$ is filtered by $\Sym$-degree, 
and the associated graded cochain complex is $\Obcl_{\sE,\sL_0}(M)$, 
whose cohomology is $\Sym(H^\bullet(\sE_{\sL_0}(M)[1]))$. 
The BV Laplacian induces on this space a differential $\Delta_{fd}$,
which arises from the $(-1)$-shifted pairing on $H^\bullet(\sE_{\sL_0}(M))$ that is induced from the corresponding pairing on the space of fields. 
This pairing respects the decompositions so that
\begin{itemize}
\item $\RR^{2g}\otimes \ker \Pi$ pairs with $\RR^{2g}\otimes \coker\, \Pi$ via the non-degenerate pairing of claim (2) of Lemma~\ref{lem: structureofcohom} and the canonical pairing between $\ker \Pi$ and $\coker \,\Pi$; 
\item $\RR^{b-1}\otimes V^\vee$ pairs with $\RR^{b-1}\otimes V$ via the duality pairing of $V$ and $V^\vee$, 
as well as the pairing of claim (2) of Lemma \ref{lem: structureofcohom}; 
\item $V$ pairs with $V^\vee$ in the canonical way. 
\end{itemize}
The Lemma then follows from direct computation of the $\Delta_{fd}$-cohomology;
in particular, see Proposition 2.4.11 of~\cite{othesis} for a proof.
\end{proof}

\subsection{Abelian CS/WZW}
\label{sec: cswzw}

As we have seen, a finite-dimensional complex vector space $\fA$ endowed with a non-degenerate symmetric pairing $\kappa$ defines an abelian Chern-Simons theory on an oriented 3-manifold $M$, with space of fields $\sE = \Omega^\bullet_{M}\otimes \fA[1]$.
We focus here on the boundary condition $\sL = \Omega^{1,\bullet}_{\Sigma}\otimes \fA$,
which encodes chiral currents, and we examine this system in three different cases of interest. 

We remark that when $\fA = \CC$, we are studying perturbative Chern--Simons theory for the abelian group ${\rm U}(1)$ and with the boundary condition given by the ${\rm U}(1)$ WZW model. 
Since we are treating Chern--Simons theory perturbatively, the integrality of the level will play no role for us. (The models we consider are isomorphic for any two nonzero values of the level.)

First, we consider the case of a compact 3-manifold $M$ with boundary and see how a Chern-Simons state of the chiral WZW system on $\del M$ arises canonically from the factorization algebra structure. 
The key observation is that the structure map for the inclusion of a tubular neighborhood $T$ of $\del M$ into $M$ induces a map 
\[
\Obq_{\sL}(\del M) \to \Obq_{\sE,\sL}(M),
\]
where the left hand side depends on the chiral currents and the right hand side ends up being, in good cases, quasi-isomorphic to ~$\CC[\hbar]$.

Second, we study the system on a manifold of the form $N\times [0,1]$, where $N$ is an oriented 2-manifold endowed with a complex structure at $t=0$ and the conjugate complex structure at $t=1$. 
Let $p: N\times [0,1]\to N$ denote the projection onto the first factor. 
We study the pushforwards $p_* \Obcl_{\sE,\sL}$ and $p_* \Obq_{\sE,\sL}$, 
which is a kind of ``slab compactification.''
In this case, we find that the pushforwards are equivalent to the factorization algebras of observables of the massless free scalar on $N$ (Lemmas \ref{lem: fullwzwscalarcl} and~\ref{lem: fullwzwscalarq}).
At the level of factorization algebras, we are recovering a ``full'' CFT by intertwining a chiral and antichiral CFT.

Finally, we study the system on a 3-manifold of the form $\Sigma \times \RR_{\geq 0}$, where $\Sigma$ is a Riemann surface.
Here, we push forward via the projection $p'$ onto $\RR_{\geq 0}$, and find that the systems are equivalent to topological mechanics on $\RR_{\geq 0}$ with values in $\Omega^\bullet(\Sigma)[1]$ and with boundary condition $\Omega^{1,\bullet}(\Sigma)$ (see Lemma~\ref{lem: cswzwonhalfline}).

\subsubsection*{Fixing the complement $L^\perp$}

The elliptic complex on $\Sigma$ perpendicular to the boundary condition $\sL$ is $\sL^\perp$, which can be identified with
\[
\sL^\perp = \Omega^{0,\bu} (\Sigma) \otimes \fA [1]
\]
equipped with the $\dbar$ differential (this is not a subcomplex of $\sEb$).
Using the obvious splitting of $\Omega^\bu(\partial M)$ into the components $\Omega^{0,\bu}(\Sigma)$ and $\Omega^{1,\bu}(\Sigma)$, we see that the differential $Q_{\partial} = \d_{dR}$ decomposes as
\[
Q_{\partial} = Q_L + Q_{L^\perp} + Q_{rel} = \dbar_{\Omega^{1,\bu}} + \dbar_{\Omega^{0,\bu}} + \partial
\]
where we view $Q_{rel} = \partial$ as the map of elliptic complexes $\partial : \Omega^{0,\bu}(\Sigma) \otimes \fA [1] \to \Omega^{1,\bu}(\Sigma) \otimes \fA$. 

\subsubsection*{The classical observables}

The sheaf $\sE_{\sL}$ of $\sL$-conditioned fields has the following explicit description.
For $U \subset M$:
\begin{align*}
\sE_{\sL} (U) & = \{\alpha \in \sE(U) \; | \; \pi(\alpha) \in \iota_* \sL(U) \} \\ & = \{\alpha \in \Omega^{\bu}(U) \otimes \fA[1] \; | \; \iota^* \alpha \in \Omega^{1, \bu} (\partial U) \otimes \fA \} .
\end{align*}
That is, the $\sL$-conditioned fields supported on $U \subset M$ consist of differential forms on $U$ whose pullback to the boundary are forms of type $(1,\bu)$. 
Likewise, we have the cosheaf $U \mapsto \sE_{\sL,c} (U)$ on $M$ which consists of compactly supported differential forms on $U$ whose pullback to the boundary are compactly supported forms of type $(1,\bu)$. 
Note that restriction here makes sense as $\iota : \partial M \hookrightarrow M$ is a closed embedding. 

The factorization algebra of classical boundary observables $\Obs^{\cl}_\sL$ on  $\Sigma$ assigns the cochain complex
\[
\Sym(\sL_c^\perp(U)) = \left(\Sym\left(\Omega_c^{0,\bu} (U) \otimes \fA [1]\right), \dbar \right) .
\]
to an open set $U \subset \Sigma$. 
Note that this is the (untwisted) enveloping factorization algebra of the cosheaf of abelian dg Lie algebras $\Omega_c^{0,\bu} \otimes \fA$ on $\Sigma$. See \S 3.6.2 of \cite{CG1}.

\subsubsection*{The quantum observables}
The factorization algebra of bulk-boundary quantum observables $\Obs^{\q}_{\sE,\sL}$ assigns to the open set $U \subset M$ the cochain complex $(\Sym(\condfieldscs[1](U))[\hbar],Q+\hbar \Delta)$.

From the general prescription in Section \ref{sec: FAconstruction} the factorization of algebra quantum boundary observables $\Obs_{\sL}^\q$ is the enveloping factorization algebra of $\sL_c^\perp[-1] = \Omega_c^{0,\bu} (U) \otimes \fA$ twisted by a local cocycle $\mu$ whose formula appears in Equation \eqref{eq: twistcocycle}.  

Since $Q_{rel}=\del$ we have the explicit formula for $\mu$: 
\[
\mu(\alpha_1,\alpha_2) = -\int_{\CC}\kappa(\alpha_1, \del \alpha_2).
\]
Explicitly, this local cocycle defines the factorization algebra on $\Sigma$ which assigns the cochain complex
\[
\left(\Sym\left(\Omega_c^{0,\bu} (U) \otimes \fA [1]\right) [\hbar] , \dbar + \hbar \mu \right) .
\]
to an open set $U \subset \Sigma$.
In other words, this is the twisted factoriazation enveloping algebra 
of the cosheaf of abelian dg Lie algebras $\Omega_c^{0,\bu}  \otimes \fA$ on $\Sigma$ associated to the local cocycle $\mu$.  
(See \S 3.6.3 of~\cite{CG1}.)

In Chapter 5 of \cite{CG1} it is shown that locally on $\Sigma = \CC$, this factorization algebra is a model for the abelian Kac-Moody vertex algebra associated to the level~$\kappa$ upon specializing $\hbar =~1$.

\subsubsection{Global sections and conformal blocks}

We now consider abelian Chern-Simons on a compact, oriented $3$-manifold $M$ coupled to chiral WZW on the boundary Riemann surface $\partial M = \Sigma$. 
In this case, $\sE_\partial(\partial M) = \Omega^\bu(\partial M) \otimes \fA$ receives two Lagrangian embeddings: 
\begin{enumerate}
\item the chiral WZW boundary condition
\[
\sL (\Sigma) = \Omega^{1,\bu}(\Sigma) \otimes \fA \hookrightarrow \sE_\partial (\partial M),
\]
\item the restriction map
\[
\pi : \sE(M) = \Omega^\bu(M) \otimes \fA[1] \to \sE_\partial (\partial M) .
\]
\end{enumerate}
When the intersection is transverse, we have $\Obcl_{\sE,\sL}(M)\simeq \CC$ and $\Obq_{\sE,\sL}(M)\simeq \CC[\hbar]$. In other words, the global sections of the factorization algebras of bulk-boundary observables are equivalent to the ground ring ($\CC$ and $\CC[\hbar]$, respectively). 

Let $T$ be a tubular neighborhood of $\del M$ in $M$. Our main theorem asserts that 
\[
\text{Obs}_{\sL}(\del M)\simeq \text{Obs}_{\sE,\sL}(T);
\]
the structure maps for the inclusion $T\subset M$ induce maps
\[
\Phi_M^{cl}:\Obcl_{\sL}(\del M)\simeq \Obcl_{\sE,\sL}(T)\to \Obcl_{\sE, \sL}(M)\simeq \CC,
\]
\[
\Phi_M^q:\Obq_{\sL}(\del M)\simeq \Obq_{\sE,\sL}(T)\to \Obcl_{\sE, \sL}(M)\simeq \CC[\hbar].
\]
We interpret $\Obcl_{\sL}(\del M)$ and $\Obq_{\sL}(\del M)$ as the spaces dual to the classical and quantum conformal blocks, respectively, of the chiral WZW model. Given a 3-manifold $M$ that cobounds $\del M$, therefore, we obtain the \emph{Chern-Simons classical (respectively quantum) conformal block} $\Phi_M^{cl}$ (respectively $\Phi_M^q$) for $M$. We can also call $\Phi_M^{cl}$ and $\Phi_M^q$ the classical and quantum Chern-Simons states for~$M$.

The idea outlined here is analogous to the discussion in Remark \ref{rmk: states}. For the case of the Poisson sigma model on an oriented, closed, connected two-manifold $N$, we were only able to define states for the boundary system when $N$ had genus 0. In the case at hand, we see that the existence of states of the sort we desire here depends on the topology of $M$ and the complex structure on $\del M$, and the way that the two interact via the inclusion $\del M\to M$. As an example, the long-exact sequence associated to the short-exact sequence
\[
0\to \condfields(M)\to \Omega^\bullet(M)[1] \to \Omega^{0,\bullet}(\del M)[1]
\]
shows that $H^1\condfields(M)$ surjects onto $H^2(M)$, so that when $H^2(M)\neq 0$, $\condfields(M)$ is not acyclic.  At least when $M$ is a handle-body, this issue does not arise, and we expect $\condfields(M)$ to be acyclic.
\subsubsection{Slab compactification}

Let $N$ be an oriented $2$-manifold.
We consider a three-manifold of the form $M= N\times [0,1]$.
Moreover, we equip $N \times \{0\}$ with a complex structure and denote it by $\Sigma$; 
we equip $N \times \{1\}$ complex conjugate complex structure and denote it by $\overline{\Sigma}$. 
Let $\iota_0$ and $\iota_1$ denote the inclusions of $N$ at $t=0$ and $t=1$, respectively. 
Let $\pi: M\to N$ be the projection onto the ``space'' slice of $M$. 
See Figure~\ref{fig:slab}.

\begin{figure}

\begin{tikzpicture}
\draw[semithick](-6,0) ellipse (0.8 and 2);
\draw[semithick](-6,0.3) arc (-30:30:1);
\draw[semithick](-6,1.2) arc (150:210:0.82);
\draw[semithick](-6,-1.3) arc (-30:30:1);
\draw[semithick](-6,-0.4) arc (150:210:0.82);
\draw[semithick](-2,0) ellipse (0.8 and 2);
\draw[semithick](-2,0.3) arc (-30:30:1);
\draw[semithick](-2,1.2) arc (150:210:0.82);
\draw[semithick](-2,-1.3) arc (-30:30:1);
\draw[semithick](-2,-0.4) arc (150:210:0.82);
\node at (-2.3,0){$\Sigma$};
\draw[semithick](2,0) ellipse (0.8 and 2);
\draw[semithick](2,0.3) arc (-30:30:1);
\draw[semithick](2,1.2) arc (150:210:0.82);
\draw[semithick](2,-1.3) arc (-30:30:1);
\draw[semithick](2,-0.4) arc (150:210:0.82);
\node at (1.7,0){$\overline{\Sigma}$};
\draw[semithick](-2,2) -- (2,2);
\draw[semithick](-2,-2) -- (2,-2);
\node at (-6.3,0){$N$};
\node at (0,-1){$\mathring{M}$};
\draw[->,semithick] (-3,0) -- (-5,0);
\node at (-4,0.25){$\pi$};
\end{tikzpicture}
\caption{Projection of $N \times [0,1]$ onto $N$}
\label{fig:slab}
\end{figure}
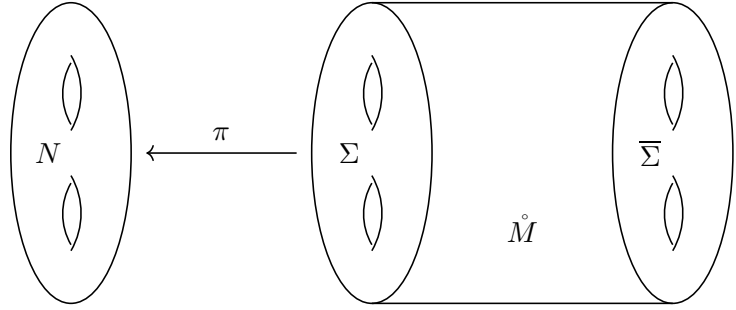

For these choices 
\[
\condfields =\left\{ \mu \in \Omega^\bullet_{N\times[0,1]}\otimes \fA[1]\,|\, \iota_0^*\mu \in \Omega^{1,\bullet}_\Sigma\otimes \fA, \iota_1^* \mu \in \Omega^{\bullet, 1}_\Sigma\otimes \fA\right\}. 
\]
is the space of $\sL$-conditioned fields for the Chern-Simons/chiral WZW bulk-boundary system.

We study now the ``slab compactification" of the factorization algebra of bulk-boundary observables.
This is the factorization algebra on $N$ obtained by pushing forward $\Obs_{\sE, \sL}$ along $\pi$. 
To decongest the notation, we assume that $\fA=\CC$, since all proofs proceed with little change for general~$\fA$.

Let $\sE_{\rm scalar}$ denote the cochain complex underlying the BV theory of the scalar field on $\Sigma$. Namely, it is the two-term chain complex
\[
\xymatrix{
\Omega^0_\Sigma \ar[r]^{\del\delbar}& \Omega^2_\Sigma
}
\]
concentrated in cohomological degrees 0 and 1, together with the natural degree --1 pairing between top forms and functions. 
Let $\Obcl_{\sE_{\rm scalar}}$ and $\Obq_{\sE_{\rm scalar}}$ denote the factorization algebras of classical and quantum obervables, respectively, for the massless free scalar.

We now show that there is a quasi-isomorphism of factorization algebras between the observables of the free scalar and the slab compactification of the bulk-boundary observables of Chern-Simons theory.
In words, this quasi-isomorphism says that the free massless 2-dimensional scalar field emerges as the theory describing a ``thin'' slab with chiral currents on one side coupled to antichiral currents on the other via a Chern-Simons theory between them.

\begin{lmm}
\label{lem: fullwzwscalarcl}
There is a quasi-isomorphism 
\[
\Obcl_{\sE_{\rm scalar}}\to \pi_*\Obcl_{\sE,\sL}
\]
of factorization algebras on $N$.
\end{lmm}

\begin{proof}
Define a map $I: \sE_{\rm scalar}\to \pi_*\sE_\sL$ by the formulas:
\begin{align*}
I(f) &= f\otimes \d t +\delbar f\otimes t -\del f\otimes (1-t)\\
I(\omega) &= \omega\otimes 1.
\end{align*}
The sheaf $\pi_*\sE_\sL$ is a subsheaf of $\Omega^\bullet_\Sigma\hotimes \Omega^\bullet_{[0,1]}([0,1])$, so we ``factorize'' forms into their tangential and normal components and write elements of $\pi_*\sE_\sL$ as tensor products.
Strictly speaking, not all forms can be written as tensor products, or even as finite sums of such. 
However, all formulas we write down have a canonical extension to the full (completed bornological) tensor product $\Omega^\bullet_\Sigma\hotimes \Omega^\bullet_{[0,1]}([0,1])$.
The map $I$ is manifestly a sheaf map, and it induces the desired quasi-isomorphism, as we proceed to show.

We construct an inverse quasi-isomorphism $P$ to $I$.
Let $f\in\Omega^0_\Sigma$, $\alpha\in \Omega^{0,1}_\Sigma$, $\beta\in \Omega^{1,0}_\Sigma$, $\omega\in \Omega^2_\Sigma$, and $\nu_1,\nu_2,\nu_3,\nu_4\in \Omega^\bullet_{[0,1]}$. 
Define the map $P: \pi_*\condfields \to \sE_{\rm scalar}$ by the formulas
\[
P(f\otimes \nu_1+\alpha \otimes \nu_2+\beta\otimes \nu_3+\omega\otimes \nu_4) =(\iota_0^* \nu_4)\omega  +f\int_{[0,1]} \nu_1-\delbar \beta \int_{[0,1]}\nu_3.
\]
Let us check that $P$ is a cochain map. Let $f\otimes \nu_1$ be a zero form on $M$ which lies in $\condfields$, i.e. it vanishes at $t=0$ and $t=1$. Then,
\[
P( (\del+\delbar)f\otimes \nu_1 +f\otimes d\nu_1) = f\otimes \nu_1(1)-f\otimes \nu_1(0)=0=Q_{\rm scalar}P(f\otimes \nu_1).
\]
Let $\alpha\otimes \nu_1+\beta\otimes \nu_2+f\otimes \nu_3$ be a one-form on $M$ which satisfies the boundary conditions to lie in $\condfields$ (here, $\nu_1,\nu_2\in \Omega^0_{[0,1]}$, and $\nu_3\in \Omega^1_{[0,1]}$; the boundary conditions are $\iota_0^*\nu_1=0$ and $\iota_1^*\nu_2=0$). Then,
\begin{align*}
P(\del\alpha \otimes \nu_1 &-\alpha\otimes d\nu_1+\delbar \beta\otimes \nu_2 -\beta\otimes d\nu_2+ (\del+\delbar)f \otimes \nu_3)\\
&= \nu_1(0)\del\alpha  + \nu_2(0)\delbar \beta + \delbar \beta (\nu_2(1)-\nu_2(0))-\delbar\del f\int_{[0,1]}\nu_3\\
&= \del\delbar f\int_{[0,1]}\nu_3 = Q\left( P(\alpha\otimes \nu_1+\beta\otimes \nu_2+f\otimes \nu_3)\right).
\end{align*}
This exhausts all the non-trivial checks that $P$ intertwines differentials.

It is immediate that $PI=\id$. 
We now construct a homotopy between $IP$ and $\id$. 
Let $\eta_0$ denote the degree --1 endomorphism of the de Rham forms $\Omega^\bullet_{[0,1]}$ which takes a one-form $\nu$ to the unique anti-derivative of $\nu$ which vanishes at $t=0$. 
Similarly, define $\eta_1$ to be the anti-derivative which vanishes at $t=1$. Now, define
\begin{align*}
K: \pi_*\condfields &\to \pi_*\condfields [-1]\\
K(f\otimes \nu_1&+\alpha\otimes\nu_2+\beta\otimes\nu_3+\omega\otimes\nu_4)
\\= &f\otimes \eta_0(\nu_1) -\left(\int_{[0,1]}\nu_1\right) f\otimes t- \alpha \otimes \eta_0(\nu_2) -\beta\otimes \eta_1(\nu_3)+ \omega \otimes \eta_0(\nu_4),
\end{align*}
One can verify by straightforward computation that $QK+KQ=IP-\id$,
which proves that $P$ and $I$ are inverse quasi-isomorphisms.

All maps involved are manifestly sheaf-theoretic over $\Sigma$, and moreover they preserve compact support. 
Hence, we also have  quasi-isomorphisms
\[
\sE_{\rm scalar,c}(U)[1]\to \condfieldscs(U\times[0,1])[1]
\]
for each $U$, and the quasi-isomorphisms respect the extension by zero maps. The lemma follows, using the usual extension of a deformation retraction between cochain complexes to a deformation retraction between the corresponding symmetric algebras.
\end{proof}

A similar lemma holds for the quantum observables.

\begin{lmm}
\label{lem: fullwzwscalarq}
There is a quasi-isomorphism
\[
\Obq_{\sE_{\rm scalar}}\to \pi_* \Obq_{\sE,\sL}
\]
of factorization algebras on~$N$.
\end{lmm}

\begin{proof}
By direct inspection, the map $I$ defined in the proof of Lemma \ref{lem: fullwzwscalarcl} respects the (--1)-shifted pairings on $\sE_{\rm scalar}$ and $\condfields$. Hence it induces also a quasi-isomorphism on the quantum observables.
\end{proof}

\begin{crl}
\label{crl: fullcswzw}
Let $\Obcl_{\chi}$ denote the boundary observables for the chiral WZW boundary condition on $\Sigma$, and similarly let $\Obcl_{\bar\chi}$ denote the boundary observables for the anti-chiral WZW boundary condition on $\overline{\Sigma}$. There is a map of factorization algebras on $N$:
\[
\Obcl_{\chi}\otimes \Obcl_{\bar \chi}\to \Obcl_{\sE_{\rm scalar}}.
\] 
There is an analogous map for the quantum factorization algebras.
\end{crl}

This map encodes the chiral and antichiral ``sectors'' of the full CFT.
When evaluated on a disk, it determines a map from a vertex algebra tensored with its conjugate into the OPE-algebra of the massless scalar field.
On a closed Riemann surface, 
the global sections of $\Obcl_\chi$ and $\Obcl_{\bar\chi}$  are (dual to) the conformal blocks of the holomorphic and anti-holomorphic Kac-Moody vertex algebras, respectively. 
This pairing of the factorization algebras gives a local-to-global description of the ``holomorphic factorization" of the conformal blocks of the full WZW theory \cite{WittenFactorization} in the case of an abelian group. 

\begin{rmk}
We now indicate how the map of Corollary \ref{crl: fullcswzw} is a quotient map, in the following sense.
Consider the factorization algebra 
\[
A=\pi_* \left( \left.\Obcl_{\sE,\sL}\right |_{N\times (0,1)}\right)
\]
on~$N$. 
Because the system under consideration is topological in the transverse direction, we may use the factorization product in the transverse direction to endow $A$ with the structure of an $E_1$-algebra in the category of factorization algebras on $N$.
Further, $\Obcl_\chi$ and $\Obcl_{\overline\chi}$ are right and left modules for $A$ (in the category of factorization algebras on $N$), respectively.
The associativity property of factorization algebras implies that the map of Corollary \ref{crl: fullcswzw} factors through the map
\[
\Obcl_\chi\otimes \Obcl_{\overline \chi}\to \Obcl_\chi\otimes_A^\LL \Obcl_{\overline \chi}.
\]
Our claim is that the resulting map
\[
\Obcl_\chi\otimes_A \Obcl_{\overline \chi}\to \Obcl_{\sE_{\mathrm{scalar}}}
\]
is an equivalence.
We make the same claim for the quantum observables.
\end{rmk}

\begin{proof}
By Theorem \ref{thm: maingenlcl}, 
\[
\Obcl_\chi \simeq \pi_*\left.\Obcl_{\sE,\sL}\right |_{N\times [0,1/2)},
\]
and
\[
\Obcl_{\bar\chi} \simeq \pi_*\left.\Obcl_{\sE,\sL}\right |_{N\times (1/2,1]}.
\]
By Lemma \ref{lem: fullwzwscalarcl}, 
\[
\Obcl_{\sE_{\rm scalar}}\simeq\pi_*\Obcl_{\sE,\sL} 
\]
The map of the present corollary is then induced from the structure maps of $\Obcl_{\sE,\sL}$ for inclusions of the form $U\times[0,1/2)\sqcup U\times (1/2,1] \subset U\times [0,1]$.
\end{proof}

\subsubsection{The other projection}

Let $M=\Sigma \times \RR_{\geq 0}$, and consider the projection $p:M\to \RR_{\geq 0}$. 
In this section, we study the pushforward factorization algebras $p_* \Obcl_{\sE,\sL}$ and~$p_* \Obq_{\sE, \sL}$,
which can be seen as studying canonical quantization of Chern-Simons theory (cf. \S4.4 of~\cite{CG1}).

Let $V=H^\bullet(\Sigma)[1]$
and endow it with the symplectic structure induced from the Poincar\'{e} duality pairing.
This graded vector space models the tangent complex at the basepoint of the $U(1)$-character stack for $\Sigma$;
its symplectic structure is also known as the Atiyah-Bott form. 
Let $L$ denote the cohomology of the holomorphic 1-forms on $\Sigma$;
it is the Lagrangian in $V$ given by the $(1,\bullet)$-part of the Dolbeault cohomology of $\Sigma$.
(A choice of K\"{a}hler metric on $\Sigma$ gives such an embedding $L\to V$.)
Conceptually picking this Lagrangian corresponds to choosing a polarization of the character stack.

By pushing forward from $M$ to $\RR_{\geq 0}$, 
we reduce the study of abelian Chern-Simons theory to the problem to the study of topological mechanics for the pair $(V,L)$,
which we treated in Section~\ref{subsec: toplmech}.

\begin{lmm}
\label{lem: cswzwonhalfline}
As factorization algebras on $\RR_{\geq 0}$, we find
\begin{itemize}
\item the classical observables $p_* \Obcl_{\sE,\sL}$ are quasi-isomorphic to $\cF_{\sO(V),\sO(L)}$, 
and
\item  the cohomology of the quantum observables $H^\bullet(p_* \Obq_{\sE,\sL})$ is isomorphic to $\cF_{W(V),F(L)}$,
which encodes the Weyl algebra associated to that tangent complex as well as the Fock space determined by the Lagrangian. 
\end{itemize}
\end{lmm}

In other words, at the classical level, our factorization algebra encodes the symplectic geometry of the $U(1$)-character stack near its base point, including the natural polarization associated with choosing a complex structure on the surface.
Our quantization recovers the canonical quantization of that data.
In short, our process recovers a shadow of the geometric quantization of abelian Chern-Simons theory.

\begin{proof}
Choose a K\"{a}hler metric on $\Sigma$. 
Let $(\sE,\sL)$ denote the Chern-Simons/chiral WZW bulk-boundary system on $\Sigma\times \RR_{\geq 0}$. 
Let $(\sF,\sK)$ denote topological mechanics on $\RR_{\geq 0}$ with values in $H^\bullet(\Sigma)[1]$ and with boundary condition $H^{1,\bullet}(\Sigma)$. 
Hodge theory using the K\"{a}hler metric allows one to construct a quasi-isomorphism
\[
\sF_{\sK,c}(U)[1]\to \condfieldscs(\Sigma\times U)[1]
\]
for any open subset $U\subset \RR_{\geq 0}$.
This quasi-isomorphism manifestly preserves the cocycles used to define the quantum observables and the extension-by-zero maps for inclusions $U\subset V$. 
It follows that  $p_* \Obcl_{\sE,\sL}$ and $p_*\Obq_{\sE,\sL}$ are equivalent to the corresponding factorization algebras for topological mechanics. 
The lemma follows via Proposition~\ref{prp: toplmech}.
\end{proof}

\subsection{Higher-dimensional abelian Chern-Simons theory}
\label{sec: higherCS}

In Example \ref{ex: highercs} we introduced a generalization of abelian Chern-Simons theory for {\em higher abelian gerbes} on oriented manifolds of dimension $4k+3$.
In physics, such theories appear in the Type IIB superstring and in the worldvolume theory of the M5 brane.
When the manifold has boundary, 
there are natural boundary conditions suggested by Hodge theory depending on a choice of a complex structure on the boundary.
For $n=0$, this boundary condition involves asking for flat connections in the bulk that become holomorphic on the boundary.
This situation has a global (or nonperturbative) refinement involving the space of holomorphic line bundles, known as the Jacobian variety of the boundary.
For $n > 0$, the global refinement of our construction here involves the {\em intermediate} Jacobian of the boundary (or, to be more accurate, a derived version thereof).
For $7$-dimensional manifolds (i.e., $k=1$), the associated boundary theory on complex $3$-folds has a close connection with holomorphic theories arising from the $M5$ brane and the $\cN = (2,0)$ superconformal theory. 
When $k=2$, there is a relationship of the theory to field strengths of the Ramond--Ramond field in Type IIB string theory.

In this section, we primarily stick to the perturbative setting, and hence work around the basepoint of the intermediate Jacobian (i.e., $\sL$ can be understood as the tangent complex at the trivial gerbe).
Our main theorems produce interesting factorization algebras on manifolds with boundary.
In particular, one could pursue analogs of the constructions we did with ordinary Chern-Simons/WZW,
and produce ``higher Chern-Simons states'' or construct a slab compactification.

\begin{rmk}
We learned of this factorization algebra associated to the intermediate Jacobian from a talk by Kevin Costello at GAP XI in Pittsburgh. 
\end{rmk}

We mostly focus here on what happens when one chooses a complex structure on the boundary which gives rise to the intermediate Jacobian.
Additionally, there is another class of boundary conditions that depends on the choice of a Riemannian metric on the boundary that we address below in Section~\ref{sec: 6d riemann}. 

\subsubsection*{The complement $L^\perp$}

We take $M$ to be a $4k+3$-manifold with boundary $\partial M$ is equipped with a complex structure.
Also, $\fA$ is a finite dimensional complex vector space equipped with a non-degenerate symmetric pairing~$\kappa$.

In the bulk $M$ we have higher dimensional abelian Chern--Simons theory whose fields are $\alpha \in \Omega^\bu_M \otimes \fA[2k+1]$ and action is
\[
S(\alpha) = \int_M \kappa(\alpha, \d \alpha) .
\]

In Example \label{ex: highercs} we introduced the intermediate Jacobian boundary condition which is defined using the complex structure on $\partial M$ by
\[
\sL = \Omega^{> k, \bullet}_{\partial M} \otimes \fA [k] . 
\]
In this case, we have 
\[
\sL^\perp= \Omega_{\partial M}^{\leq k, \bullet}\otimes\fA[2k+1]
\]
(the grading is such that forms have cohomological degree given by their form degree minus $2k+1$). The twisting cocycle $\mu$ is given by 
\[
\mu(\alpha_1,\alpha_2)=-\int_{\partial M}\kappa(\alpha_1,\del \alpha_2);
\]
it is non-zero only on $\Omega_{\partial M}^{k,\bullet}\otimes \fA[2k+1]$. 

\subsubsection*{The quantum observables}
The factorization algebra of quantum observables $\Obs^{\q}_{\sL}$ assigns to the open set $U \subset M$ the cochain complex $(\Sym(\condfieldscs[1](U))[\hbar],Q+\hbar \Delta)$.

From the general prescription in Section~\ref{sec: FAconstruction}, 
the factorization algebra of quantum boundary observables $\Obs^\q_{\sL}$ is the enveloping factorization algebra of $\sL_c^\perp [-1] = \Omega_c^{\leq k,\bu} (U) \otimes \fA[2k]$ twisted by a local cocycle $\mu$ whose general formula appears in equation~\eqref{eq: twistcocycle} is given by
\[
\mu(\alpha_1,\alpha_2) = -\int_{U}\kappa(\alpha_1, \del \alpha_2)
\]
where $\alpha_1, \alpha_2 \in \Omega^{\leq k, \bu}_c (U)$, for $U \subset \partial M$ an open set.
Explicitly, this local cocycle defines the factorization algebra on $\partial M$ which assigns the cochain complex
\begin{equation} \label{eqn:higherbo}
\left(\Sym\left(\Omega_c^{\leq k,\bu} (U) \otimes \fA [2k+1]\right) [\hbar] , \dbar + \partial + \hbar \mu \right) 
\end{equation} 
to the open subset $U\subset \partial M$
In other words, this is the twisted factoriazation enveloping algebra 
of the cosheaf of abelian dg Lie algebras $\Omega_{\partial M , c}^{\leq k,\bu} \otimes \fA[2k]$ associated to~$\mu$. 

\subsubsection{Relationship with the intermediate Jacobian}

Let $X$ denote the complex manifold $\partial M$. 
We have seen that the boundary observables of higher dimensional abelian Chern-Simons theory are obtained as the (twisted) enveloping factorization  algebra of the abelian dg Lie algebra $\fJ_X := \Omega^{\leq n, \bu}(X)[2k]$.
When $X$ is compact K\"{a}hler, this dg Lie algebra is related to the {\em intermediate Jacobian}~\cite{Griffiths1, Griffiths2, ClemensGriffiths}.

Like our theories and boundary conditions have been,
notice that the object $\fJ_X$ is sufficiently local on the complex manifold $X$, meaning it is given as the sheaf of sections of a vector bundle and its differential is a differential operator.  
Although no Lie bracket plays a role here, we will continue to refer to it as a sheaf of abelian dg Lie algebras.

For $X$ compact and K\"{a}hler, the $(k+1)$st intermediate Jacobian is defined by
\[
J_{k+1}(X) = H^{2k+1} (X, \CC) / (F^{k+1} H^{2k+1}(X, \CC) \oplus H^{2k+1}(X , \ZZ)) .
\]
where $F^{k+1} H^k(M, \CC) = \oplus_{i > k} H^{i,j-i} (X)$ is the $(k+1)$st step in the Hodge filtration.
There is a canonical isomorphism $H^{2k+1}(X, \RR) = H^{2k+1}(X, \CC) / F^{k+1} H^k(M, \CC)$, so one can also identify the intermediate Jacobian with $H^{2k+1}(X , \RR) / H^{2k+1}(X , \ZZ)$. 
The group $H^{2k+1}(X, \ZZ)$ forms a lattice inside of $H^{2k+1}(X, \CC) / F^{k+1} H^k(M, \CC)$, hence the intermediate Jacobian is a complex torus of dimension half of the $(2k+1)$st Betti number of $X$. 
When $k=0$, so $X = \Sigma$ is a Riemann surface, the intermediate Jacobian is the usual Jacobian variety. 

The tangent space of the intermediate Jacobian can be identified with the cohomology group $H^{k, k+1}(X)$. 
That is, it is precisely $H^1$ of the dg Lie algebra $\fJ_X$.
When $k=0$, the dg Lie algebra describing its infinitesimal behavior is simply $\fJ_\Sigma = \Omega^{0,\bu}(\Sigma)$, 
which describes deformations of the trivial holomorphic line bundle. 

We can regard the dg Lie algebra $\fg_X$ as a derived enhancement for the formal neighborhood of a point in $J_{k+1}(X)$ (see Section 9 of \cite{FMabeljacobi}),
so we refer to the (sheaf of) dg Lie algebra(s) $\fJ_X$ as the ``intermediate Jacobian dg Lie algebra.''
We anticipate that a global derived intermediate Jacobian exists whose tangent complex is indeed modeled by $\fJ_X$. 
We also expect that it leads to a factorization {\em space}, providing an analog of the Beilinson-Drinfeld Grassmannian, 
but the techniques needed to construct it are quite different than those we use in this paper.

\subsubsection{Relationship to physics}\label{sec:physics}

In physics, higher dimensional Chern-Simons theory has appeared in various contexts. 
Seven-dimensional Chern-Simons theory and the intermediate Jacobian on complex three-folds have appeared in the context of $M$-theory \cite{WittenM5effective}. 
Specifically, the seven-dimensional Chern--Simons theory is holographically dual to the ``chiral'' two-form that lives in the worldvolume theory of the $M5$-brane. 
Another instance arises from Ramond--Ramond fields in the Type IIB superstring in ten dimensions, 
which are described by ``chiral'' four-forms whose field strengths are required to be self-dual.

Throughout this section we focus on the {\em holomorphic twist} of the supersymmetric theories in question.
Such a holomorphic twist of a supersymmetric theory is a sector of the full theory that depends holomorphically on a complex structure placed on spacetime,
much as a topological twist (when it exists) is a sector that depends on the underlying smooth structure (see \cite{CostelloHolomorphic} for the mathematical development).
Holomorphic twists of theories are appealing mathematically as their moduli spaces admit elegant descriptions in terms of complex geometry, 
as we see here with the intermediate Jacobian, 
but our methods can also be used to study these theories before twisting.
(See Section~\ref{sec: 6d riemann} for comments on the untwisted setting.)

First, let's consider seven-dimensional Chern-Simons theory and describe how our setting is related to the physical one. 
The relevant physical theory is the abelian six-dimensional $\cN = (2,0)$ superconformal theory. 
This theory is rather peculiar as it contains among its field content a two-form with a self-dual field strength that admits no Lagrangian formulation.
In physics, such a two-form is referred to as a ``chiral'' two-form.
Physical arguments identify this six-dimensional theory as the low energy effective theory of a single $M5$-brane~\cite{WittenM5effective, Schwarz}.

The holomorphic twist of the abelian $\cN = (2,0)$ superconformal theory on $\RR^6$ has been analyzed in the BV formalism by Saberi and the third author~\cite{SW3}.
They show that the twist decomposes as the product of a free holomorphic theory and the intermediate Jacobian theory we described above. 
More precisely, they prove the following.

\begin{thm}\cite[Theorem 4.2]{SW3}
For the abelian $\cN = (2,0)$ superconformal theory on $\RR^6$ formulated in the BV formalism,
the holomorphic twist of the chiral two-form component is perturbatively equivalent to the theory of the intermediate Jacobian described by the dg Lie algebra $\fJ_{\CC^3} = \Omega^{\leq 1, \bu}(\CC^3)[1]$. 
\end{thm}

Notice that Maurer--Cartan elements in the dg Lie algebra $\fJ_{\CC^3}$ consist of pairs $\alpha \in \Omega^{0,2}(\CC^3)$ and $\beta \in \Omega^{1,1}(\CC^3)$ that satisfy $\dbar \alpha = 0$ and $\dbar \beta + \partial \alpha = 0$. 
The fields $\alpha, \beta$ are the components of the chiral two-form that survive in the holomorphic twist.
This theorem uses a description of the abelian $\cN = (2,0)$ theory in the BV formalism that was, in part, motivated by the work~\cite{ESW}.
Another description of twists of the tensor multiplet can be found in~\cite{Cederwall}.

In \cite{WittenM5effective}, an analogy is formulated between the untwisted self-dual theory on three-folds and the theory of the chiral boson in one complex dimension vis-\`{a}-vis their relationships to Chern-Simons theory.
In complex dimension one, we have witnessed this analogy at the level of factorization algebras in Section~\ref{sec: cswzw}: the factorization algebra of abelian Chern-Simons theory on three-manifolds with boundary is equivalent to the factorization algebra of the ${\rm U}(1)$-valued chiral boson on the boundary.
In complex dimension three, 
we propose a holomorphic variant of this setup with the dg Lie algebra $\fJ_X = \Omega^{\leq k, \bu}(X)[2]$ of the intermediate Jacobian sitting at the boundary of seven-dimensional Chern-Simons theory. 

Indeed, by the above theorem, the factorization algebra of classical observables for the holomorphic twist has a factor arising from the intermediate Jacobian, namely,
\[
\left(\Sym\left(\Omega_{\CC^3, c}^{\leq 1,\bu} [3]\right), \dbar + \partial \right),
\]
and this factor is precisely the $\hbar \to 0$ limit in~\eqref{eqn:higherbo} of the factorization algebra of boundary observables for seven-dimensional Chern-Simons theory (when $\fA = \CC$) with our preferred holomorphic boundary condition.
As a corollary of our work here we thus obtain a quantization of these classical observables. 

\begin{crl}
Bulk-boundary quantization provides a BV quantization of the classical observables of the abelian $\cN=(2,0)$ superconformal theory in the holomorphic twist. 
\end{crl}

Note that the free theory component of the holomorphic twist admits an easy BV quantization,
so that we have obtained a quantization of the holomorphic twist of the abelian 6d theory.
In future work we plan to study aspects of this quantization in more detail as it relates to six-dimensional superconformal field theories. 

Next, we turn to 11-dimensional Chern-Simons theory and its relationship to the Type IIB superstring. 
Costello and Li \cite{CostelloLiSUGRA} have proposed a model for the holomorphic twist of the Type IIB supergravity on a Calabi--Yau five-fold $X$, which includes fields described by the cochain complex
\[
\PV^{2,\bu}(X) \xto{\partial_\Omega} \Pi \PV^{1,\bu}(X) \xto{\partial_\Omega} \PV^{0,\bu}(X), 
\]
in its field content.
Here, $\PV^{j,\bu} (X) = \Omega^{0,\bu}(X, \wedge^i T^{1,0} X)$ denotes the Dolbeault complex of $j$-polyvector fields and $\partial_\Omega$ is the divergence operator with respect to the holomorphic volume form.
The $\Pi(-)$ denotes parity shift as this theory only makes sense in a $\ZZ/2$-graded sense, 
so we forget any $\ZZ$-gradings down to $\ZZ/2$-gradings.

Using the holomorphic volume form, one can identify this cochain complex with the cochain complex $\Omega^{\geq 3, \bu}(\CC)$. 
Within this complex is a field of Dolbeault form type $(3,2)$, which is identified with a piece of the five-form field strength of Ramond-Ramond field in the original untwisted theory. 
The factorization algebra of classical observables is
\[
\left(\Sym\left(\Pi \Omega_{X, c}^{\leq 2,\bu} \right), \dbar + \partial \right), 
\]
which is  the $\hbar \to 0$ (and $\ZZ/2$ degeneration) of the complex of boundary observables of 11-dimensional Chern-Simons theory with our preferred holomorphic boundary condition.
For a study of the BV quantization of the full interacting Type IIB theory in the holomorphic twist, we refer to~\cite{CLbcov1}. 

\subsubsection{Compactification}

In this section we show that the intermediate Jacobian is closely related to a familiar object in two-dimensional chiral conformal field theory.
Let $X$ be a complex $(2k+1)$-fold of the form $\Sigma \times \CC P^{2k}$. 

\begin{lmm}
Let $X = \Sigma \times \CC P^{2k}$, where $\Sigma$ is a Riemann surface, and let $\pi : X \to \Sigma$ be the projection. 
Then $\pi_* \fJ_X$ is equivalent to the sheaf of dg Lie algebras 
\[
\Omega^{0,\bu}_\Sigma \oplus \left( \bigoplus_{j=1}^{n-1} \underline{\CC}[2k+2j] \right) 
\]
on $\Sigma$,
where $\underline{\CC}[q]$ is the constant sheaf of $\Sigma$ with one-dimensional fiber concentrated in degree $-q$. 
\end{lmm}

\begin{proof}
By formality of projective space we have an equivalence of sheaves of dg Lie algebras on~$\Sigma$:
\[
\pi_* \fJ_X \simeq \bigoplus_{i+j \leq n} \Omega^{i,\bu}(\Sigma) \otimes H^{j,\bu}(\CC P^{2k}) [2k-2] .
\]
The only remaining differential is $\dbar_\Sigma + \partial_\Sigma$. 

For $j < k$, this complex is a direct sum of complexes of the form
\[
\begin{tikzcd}
\underline{-2k - 2j} & \underline{-2k-2j+1} \\
\Omega^{0,\bu}(\Sigma) \otimes H^{j,j}(\CC P^{2k}) \ar[dr, "\partial_\Sigma"] &  \\
& \Omega^{1,\bu}(\Sigma) \otimes H^{j,j}(\CC P^{2k})  .
\end{tikzcd}
\]
The remaining part of the complex is $\Omega^{0,\bu}(\Sigma) \otimes H^{n,n}(\CC P^{2k}) = \Omega^{0,\bu}(\Sigma)$. 
Thus, $\pi_* \fJ_X$ is quasi-isomorphic to 
\[
\Omega^{0,\bu}(\Sigma) \oplus \left( \bigoplus_{j=1}^{n-1} \underline{\CC}[2k+2j] \right) 
\]
as a sheaf of Lie algebras on~$\Sigma$.
\end{proof}

Next, we show that the central extensions are compatible. 
The boundary condition of higher dimensional Chern-Simons theory we discussed in Section \ref{sec: higherCS} gives rise to a factorization algebra of boundary observables $\Obs^{\q}_{\sL}$ on the boundary complex $(2k+1)$-fold $X$. 
For now, denote this factorization algebra by $\cF_{X, \kappa}$. 
This factorization algebra is the enveloping factorization algebra of $\fJ_X$ twisted by the local cocycle $\mu (\alpha, \beta) = \int_X \kappa(\alpha \partial \beta)$.

On the $(2k+1)$-fold $X = \Sigma \times \CC P^{2k}$, there is the following relationship between the factorization algebras $\cF_{X,\kappa}$ and $\cF_{\Sigma,\kappa}$. 
Note that $\cF_{\Sigma,\kappa}$ is the factorization algebra at the boundary of ordinary Chern-Simons theory on a $3$-manifold with boundary $\Sigma$;
that is, it is the $U(1)$ Kac-Moody factorization algebra at level~$\kappa$.

\begin{crl}
Let $\pi : \Sigma \times \CC P^{2k} \to \Sigma$.
There is a map of factorization algebras on $\Sigma$:
\[
\pi_* \cF_{X,\kappa} \to \cF_{\Sigma, {\rm vol}(\CC P^{2k}) \kappa}  .
\]
\end{crl}

What we have shown is that $\pi_* \cF_{X, \kappa} = \cF_{\Sigma, {\rm vol}(\CC P^{2k}) \kappa}  \otimes \cG$ where $\cG$ is a locally constant factorization algebra on $\Sigma$, which is independent of $\kappa$ and the volume of~$\CC P^n$. 

Let's briefly put these observations in the context of this paper.
Let $\RR_{\geq 0} \times \Sigma \times \CC P^{2k}$ be a $4k+3$-dimensional manifold with boundary $X$,
and equip it with the higher abelian CS/WZW system we've just described.
We can compactify this whole system along $\CC P^{2k}$ to get a bulk-boundary system on $\RR_{\geq 0} \times \Sigma$.
We have just seen that the boundary observables look like a chiral current algebra tensored with a locally constant factorization algebra that depends on the topology of $\CC P^{2k}$.
In more conventional terminology, it's a chiral CFT coupled trivially to a 2d TFT.
The bulk observables behave similarly.
For the higher dimensional Chern-Simons theory on the bulk $\RR_{> 0} \times \Sigma \times \CC P^{2k}$,
the factorization algebra of bulk observables pushes forward to $\RR_{> 0} \times \Sigma$.
There, it looks like the observables of a 3-dimensional abelian Chern-Simons theory with values in the graded abelian Lie algebra~$H^*(\CC P^{2k})[2k+1]$.

\subsubsection{A Riemannian variation}
\label{sec: 6d riemann}

We briefly discuss the Riemannian boundary condition of higher dimensional Chern--Simons that depends on a Riemannian metric of Example~\ref{ex: Riemhighercs}.
This Riemannian version of the boundary condition can be used to treat the examples in Section~\ref{sec:physics} {\em before} performing the holomorphic twist, i.e., with the original supersymmetric theory.

Recall that $\Omega^\bu_{\partial M} \otimes \fA [2k+1]$ has a subcomplex
\[
\sL = \bigg(\Omega^{2k+1}_+(N) \otimes \fA \xrightarrow{\d} \Omega^{2k+2} (N) \otimes \fA [-1] \xrightarrow{\d} \cdots \xrightarrow{\d} \Omega^{4k+2}(N) \otimes \fA [-2k-1]\bigg) 
\]
by using the decomposition of the middle de Rham forms under the Hodge $\star$ operator.
The elliptic complex $\sL^\perp$ on $N$ ``perpendicular'' to the boundary condition $\sL$ can be identified with
\begin{equation}\label{eqn:dminus}
\sL^\perp = \bigg( \Omega^{0} (N) \otimes \fA [2k+1].\xto{\d} \Omega^1(N) \otimes \fA[2k] \to \cdots \to \Omega^{2k} (N) \otimes \fA [1] \xto{\d_-} \Omega^{2k+1}_- (N) \otimes \fA \bigg)
\end{equation}
where $\d_- : \Omega^{2k} (N) \to \Omega^{2k+1}_- (N)$ denotes the de Rham differential followed by the projection using the decomposition~(\ref{eqn:decomp}). 
In turn, we can read off the factorization algebra of classical boundary observables $\Obs^\cl_{\sL}$ which assigns to the open set $U \subset N$, 
the cochain complex~$\Sym(\sL_c^\perp(U))$. 

The factorization algebra of quantum boundary observables $\Obs^\q_{\sL}$ is the enveloping factorization algebra of $\sL_c^\perp [-1]$ twisted by a local cocycle $\mu$ whose formula appears in equation~\eqref{eq: twistcocycle}. 
A similar calculation as in the complex case reveals an explicit formula for $\mu$: 
\[
\mu(\alpha_1,\alpha_2) =- \int_{U}\kappa(\alpha_1, \d \alpha_2),
\]
where $\alpha_1, \alpha_2 \in \sL^\perp_c (U)$ for $U \subset N$ an open set.
This local cocycle defines explicitly the factorization algebra of quantum boundary observables on $N$ that assigns the cochain complex
\[
\left(\Sym\left(\sL_c^\perp (U) \right) [\hbar] , \d + \d_- + \hbar \mu \right) 
\]
where $\d + \d_-$ denotes the differential in Equation~(\ref{eqn:dminus}). 

\appendix

\section{A lemma of Atiyah-Bott type}

For an elliptic complex on a manifold without boundary, 
the complex of compactly supported smooth sections embeds into the complex of compactly supported distributional sections.
The {\it Atiyah-Bott lemma} is that this embedding is a continuous quasi-isomorphism 
(see Appendix D of \cite{CG1}),
and it plays a role in constructing the observables of free BV theories.
We wish to prove an analog relevant to free bulk-boundary theories.

Let $(\sE,\sL)$ be a free bulk-boundary system. 
We use the pairing $\ip$ to view $\condfieldscs[1]$ as a space of linear functionals on $\condfields$: 
each section $e_1\in \condfieldscs[1]$ gives a linear functional $\Phi_{e_1}$ by the formula
\[
\Phi_{e_1}(e_2)=\ip[e_1,e_2].
\]
This embedding has the following property.

\begin{prp}
The map $\Phi_\cdot$ induces a quasi-isomorphism of complexes of cosheaves
\[
\condfieldscs[1]\to \condfields^\vee,
\]
where $\condfields^\vee$ is the cosheaf which assigns to the open $U$, the strong topological dual to~$\condfields(U)$.
More precisely, on each open $U$, this map is a continuous linear map of topological vector spaces and a quasi-isomorphism.
\end{prp}

\begin{proof}
The map is continuous because it is the composite 
\[
\condfieldscs[1]\hookrightarrow \sE_c[1]\to \sE^\vee \to \condfields^\vee.
\]
The map preserves the differential $Q$ because 
\[
\Phi_{Q e_1}(e_2)=\ip[Qe_1,e_2]=\pm \ip[e_1,Qe_2]=\Phi_{e_1}(Qe_2);
\]
this is only true because we have imposed the boundary condition $\sL$. It manifestly respects the extension maps of cosheaves. It remains only to check that it is a quasi-isomorphism. In the proof of Theorem \ref{thm: maingenlcl}, we show that $\condfieldscs[1]$ is a homotopy cosheaf; an almost identical argument shows that $\condfields^\vee$ is also a homotopy cosheaf. Hence, given any open $U\subset \del M$, and any (locally finite) cover $\fU$ of $U$, we have the following commutative diagram
\[
\begin{tikzcd}
\check{C}(\condfieldscs[1],\fU)\ar[r,"\sim"]\ar[d,]&\condfieldscs(U)[1]\ar[d]\\
\check{C}(\condfields^\vee,\fU)\ar[r,"\sim"]&\condfields^\vee(U)
\end{tikzcd}.
\]
We will show that the left-hand downward pointing map is a quasi-isomorphism. 

Fix a tubular neighborhood $N\cong \del  M\times [0,T)$ of $\del M$. 
Let us assume that the cover $\fU$ is ``somewhat nice:'' 
it consists of open subsets $U_\alpha$ such that either $\overline{U}_\alpha\cap \del M=\emptyset$ or $V_\alpha\subset N$ of the form $V_\alpha\cong V'_\alpha\times [0,T')$ where $V'_\alpha$ is an open set in $\del M$. 
All finite intersections of somewhat nice sets are also somewhat nice, 
so all the summands in the \v{C}ech complexes will be of the form $\condfieldscs[1](U')$ or $\condfields^\vee(U')$ for $U'$ somewhat nice. 
If we prove that the map $\condfieldscs[1](U')\to \condfields^\vee(U')$ is a quasi-isomorphism for $U'$ somewhat nice, 
then the proposition follows, since the \v{C}ech complex has a filtration by degree of intersection 
(which is preserved by the map $\check{C}(\condfieldscs[1],\fU)\to \check{C}(\condfields^\vee,\fU)$) 
and the induced map on the associated graded spaces is a sum of maps $\condfieldscs[1](U')\to \condfields^\vee(U')$ for $U'$ somewhat nice. 

If $\overline{U'} \cap \del M=\emptyset$, then the map $\condfieldscs[1](U')\to \condfields^\vee(U')$ is a quasi-isomorphism, by the Atiyah-Bott lemma (see Appendix D of \cite{CG1}). Otherwise, suppose $U'=V\times [0,T')$, and let $L':=\Eb/L$. Denote by $\sL^\perp$ the sheaf of sections of $L'$. We saw in the proof of Theorem \ref{thm: maingenlcl} that there is a deformation retraction of $\condfieldscs[1](U')$ onto $\sL^\perp(V)$. Similarly, there is a deformation retraction of $\condfields(U')$ onto $\sL(V)$, and hence of $\sL^\vee(V)$ onto $\condfields^\vee(U')$. 
The map $\condfields^\vee(U')\to \sL^\vee(V)$ in this deformation retraction is dual to the inclusion $\sL(V)\to \condfields(U')$ of the $\sL$ fields as constants in the normal direction. From the characterization of the map $\sL^\perp_c(V)\to \condfieldscs[1](U')$ in Theorem \ref{thm: maingenlcl}, it follows that the composite 
\[
\sL^\perp_c(V)\to \condfieldscs[1](U')\to \condfields^\vee(U')\to \sL^\vee(V)
\]
is the Atiyah-Bott quasi-isomorphism (using the pairing $\ip_\partial$ to identify $L'$ with $L^!$). It follows that the map $\condfieldscs(U')\to \condfields^\vee(U')$ is a quasi-isomorphism, whence the proposition. 
\end{proof}

\section{Topological tensor products in the presence of boundary conditions}

How to find the correct ``natural'' tensor product of topological vector spaces is a notoriously subtle question.
In some situations there are options that are appealing for several reasons.
For instance, given two vector bundles $V_1\to M_1$ and $V_2\to M_2$, 
let $\sV_1$ and $\sV_2$ denote the locally convex topological vector spaces consisting of the smooth global sections of $V_1$ and $V_2$, respectively. 
There is a standard isomorphism (of topological vector spaces)
\[
\sV_{1}\hotimes_\pi \sV_2 \cong \cinfty(M_1\times M_2, V_1\boxtimes V_2),
\]
where $V_1\boxtimes V_2$ is the external tensor product of the bundles $V_1$ and $V_2$
and $\hotimes_\pi$ is the completed projective tensor product of locally convex topological vector spaces.
In this case the geometrically attractive answer matches a completion that is natural from functional analysis.
We use that fact---and a compactly supported analog---in defining the observables of a free BV theory on a manifold without boundary (see Section~\ref{sec: bulk obs}). 
Its main technical role is in the proof that the observables form a factorization algebra.

When we work with free bulk-boundary theories, we would like a similar geometric understanding of the completed bornological tensor product $\hotimes_\beta$. 
From the point of view of the paper, 
this appendix is devoted to proving that $(\condfieldscs[1](U))^{\hotimes_\beta k}$
is isomorphic to the space of compactly-supported sections of $E^{\boxtimes k}$ over $U^{\times k}$ whose $j$th tensor factor lies in $L\oplus \Eb \,\d t$ when the corresponding $M$ coordinate lies on~$\del M$. 
But it is natural to treat several generalizations and variants of this fact. 

\textit{The discussion here is highly technical, and its main technical role is in the proof that the bulk-boundary observables form a factorization algebra. 
This section is not needed unless the reader wants a detailed understanding of the vector spaces appearing in the bulk-boundary observables.}

To state these generalizations, let $M_1,\cdots, M_k$ be manifolds with boundary, $V_1\to M_1, \cdots, V_k\to M_k$ be vector bundles on the $M_i$, and $W_1\subset V_1\big|_{\partial M_1},\cdots, W_k\subset V_k\big|_{\partial M_k}$ be subbundles of the indicated bundles. Breaking slightly with our usual notation, we will let $\sV_i:= \cinfty(M, V_i)$ and $\sW_i:= \cinfty(\del M, W_i)$,
i.e. we use the script letters to denote the spaces of global sections of vector bundles instead of the corresponding sheaves of sections.
\begin{ntn}
Define 
\[
(\sV_i)_{W_i}:= \{ \sigma \in \sV_i \mid \sigma|_{\del M_i} \in \cinfty (M_i,W_i)\}.
\]
The space $(\sV_i)_{W_i}$ is a closed subspace of $\sV_i$; since the latter space is nuclear Fr\'{e}chet, the former is as well. More categorically, $(\sV_i)_{W_i}$ is the pullback
\[
\begin{tikzcd}
(\sV_i)_{W_i} \ar[r]\ar[d]& \sV_i\ar[d]\\
\sW_i \ar[r]& \cinfty(\del M, V_i\mid_{\del M})
\end{tikzcd}
\]
\end{ntn}

\begin{ntn}
Define
\[
\tp:= \cinfty(M_1\times\cdots\times M_k, V_1\boxtimes \cdots \boxtimes V_k),
\]
and
\begin{align*}
\tpbc&:=\\
\big \{ \sigma \in &\tp \mid \sigma (x_1,\cdots, x_k)\in (V_1)_{x_1}\otimes \cdots \otimes (W_i)_{x_i}\otimes \cdots \otimes (V_k)_{x_k} \text{ when } x_i\in \partial M_i \big\};
\end{align*}
in other words, $\tpbc$ consists of sections of $V_1\boxtimes \cdots \boxtimes V_k$ whose $i$-th tensor factor belongs to $W_i$ whenever the corresponding coordinate lies in $\del M_i$. We endow $\tpbc$ with the topology which it inherits as a subspace of $\tp$. The resulting locally convex topological vector space $\tpbc$ is nuclear Fr\'{e}chet, since it is a closed subspace of $\tp$. $\tpbc$ can be described as a limit in the category of topological vector spaces. More precisely, it is the simultaneous limit of all the diagrams of the form 
\[
\begin{tikzcd}
&\cinfty(M_1\times \cdots \times\del M_i\times\cdots \times M_k, V_1\boxtimes\cdots \boxtimes W_i\boxtimes \cdots \boxtimes V_k)\ar[d]\\
\tp\ar[r]& \cinfty(M_1\times \cdots \times \del M_i \times \cdots\times M_k, V_1\boxtimes \cdots\boxtimes (V_i)\mid_{\del M_i}\boxtimes \cdots \boxtimes V_k).
\end{tikzcd}
\]
as $i$ ranges from 1 to $k$. 
\end{ntn}

Note that the continuous multilinear map 
\[
\sV_1\times\cdots\times \sV_k \to \tp,
\]
when restricted to $(\sV_1)_{W_1}\times \cdots \times (\sV_k)_{W_k}$, has image in $\tpbc$, so there is a natural map 
\[
\fS : (\sV_1)_{W_1}\widehat{\otimes}_\pi \cdots \widehat{\otimes}_\pi (\sV_k)_{W_k}\to \tpbc.
\]

We can establish similar notations when we require compact support for sections of the $V_i$. Let us choose compact subsets $\cK_i\subset M_i$. We choose to use a calligraphic font for the $\cK_i$ because the symbols $\cK_i$ and $W_i$ will both appear in subscripts in our notation, and we want to make clear that the two subscripts serve different purposes. 

\begin{ntn}
Let 
\begin{enumerate}
\item  $(\sV_i)_{\cK_i}$ denote the space of sections of $V_i$ with compact support on $\cK_i$;
\item $(\sV_i)_{\cK_i,W_i}$ denote the space
\[
(\sV_i)_{\cK_i}\cap (\sV_i)_{W_i},
\]
i.e. $(\sV_i)_{\cK_i,W_i}$ is the space of sections of $V_i$ satisfying both a boundary condition and a compact support condition;
\item $\tpcs$ denote the subspace of $\tp$ consisting of sections with compact support on $\cK_1\times \cdots \times \cK_k$; and
\item $\tpbccs$ denote the space
\[
\tpcs \cap \tpbc.
\]
As with $(\sV_i)_{\cK_i,W_i}$, the sections in $\tpbccs$ satisfy both a boundary condition and a compact support condition.
\end{enumerate}
All four spaces are nuclear Fr\'{e}chet spaces.
\end{ntn}

There is a map 
\[
\fS_{c.s.} : (\sV_1)_{W_1,\cK_1}\widehat{\otimes}_\pi \cdots \widehat{\otimes}_\pi (\sV_k)_{W_k,\cK_k}\to \tpbccs.
\]

The aim of this appendix is to prove the following result.

\begin{thm}
\label{thm: tensorofdirichlet}
The maps $\fS$ and $\fS_{c.s.}$ are isomorphisms for the topological vector space structures.
\end{thm}

\begin{proof}
The completed projective tensor product commutes with limits separately in each variable. Hence, 
\[
(\sV_1)_{W_1}\widehat{\otimes}_\pi \cdots \widehat{\otimes}_\pi (\sV_k)_{W_k}
\]
can be identified with the simultaneous limit of diagrams of the form 
\[
\begin{tikzcd}
&\sV_1\hotimes_\pi \cdots \hotimes_\pi \sW_i \hotimes_\pi \cdots \hotimes_\pi\sV_k  \ar[d]\\
\sV_1\hotimes_\pi \cdots\hotimes_\pi \sV_k\ar[r]& \sV_1\hotimes_{\pi}\cdots\hotimes_\pi \cinfty(\del M, V_i\mid_{\del M}) \hotimes_\pi \cdots \hotimes_\pi \sV_k.
\end{tikzcd}
\]
as $i$ ranges from $1$ to $k$. As we have seen, $\tpbc$ is a similar limit. The isomorphism $\sV_1 \hotimes_\pi \cdots \hotimes_\pi \sV_k \to \tp$ and its analogs for the other entries of the diagrams induces an isomorphism between the diagram defining $\sV_1\hotimes_\pi \cdots \hotimes_\pi \sW_i \hotimes_\pi \cdots \hotimes_\pi\sV_k$ and the one defining $\tpbc$. $\fS$ is induced from this isomorphism of diagrams, so is an isomorphism. The same exact argument applies for~$\fS_{c.s}$.
\end{proof}

We now describe a consequence of Theorem~\ref{thm: tensorofdirichlet} that is of more direct relevance to the present context. Let us momentarily suppress the $i$ subscripts from our notation, letting $V\to M$ be a vector bundle and $W$ a subbundle of $V\mid_{\del M}$. We define $(\sV)_{W,c}$ to be the space
\[
\text{colim}\left(  (\sV)_{W,\cK_{1}}\to (\sV)_{W,\cK_2}\to\cdots\right),
\]
with $\cK_{j}\subset \cK_{(j+1)}$ and $\cup_{j} \cK_{j} = M$, i.e. the $\cK_j$ form a sequence of compact subsets of $M$ exhausting it. Equivalently, we can define $(\sV)_{W,c}$ via the pullback diagram
\[
\begin{tikzcd}
(\sV)_{W,c}\arrow[r]\arrow[d] \arrow[rd, phantom, "\lrcorner", at start] & (\sV)_c\arrow[d]\\
(\sV)_{W}\arrow[r,hook] & \sV
\end{tikzcd};
\]
here $\sV_c$ is the space of compactly-supported sections of $V$ endowed with the inductive limit topology (when $\sV_c$ is endowed with this topology, the arrow on the right-hand side of the above diagram is not an embedding). The completed projective tensor product does not commute with colimits; hence Theorem \ref{thm: tensorofdirichlet} does not help us to compute completed projective tensor products of spaces of the form $(\sV)_{W,c}$. We may, however, forget the topology of all spaces involved, remembering only the bounded subsets. In other words, we remember only the underlying bornological vector spaces. Once we do, a new tensor product becomes available to us, namely the completed \emph{bornological} tensor product. The completed bornological tensor product \emph{does} commute with colimits. For nuclear Fr\'{e}chet spaces, it coincides with the completed projective tensor product. See \S B.4-5 of \cite{CG1} for details. 

In the main body of the text, we always use the completed bornological tensor product. Hence, we need to use Theorem \ref{thm: tensorofdirichlet} to infer statements about the bornological tensor products of interest to us. This task is undertaken in the following corollary:

\begin{crl}
\label{crl: borntensorofdirichlet}
There are isomorphisms of bornological vector spaces
\[
(\sV_1)_{W_1}\,\hotimes_\beta\, \cdots \,\hotimes_\beta\,(\sV_k)_{W_k}\cong \tpbc,
\]
\[
(\sV_1)_{W_1,c}\,\hotimes_\beta\,\cdots\, \widehat{\otimes}_\beta\, (\sV_k)_{W_k,c}\cong \tpbccslf
\]
Here, $\tpbccslf$ is defined analogously to~$(\sV_i)_{W_i,c}$.
\end{crl}

\begin{proof}[Proof of Corollary]
The isomorphism 
\[
(\sV_1)_{W_1}\,\hotimes_\beta\, \cdots \,\hotimes_\beta\,(\sV_k)_{W_k}\cong \tpbc
\]
is a direct consequence of Theorem \ref{thm: tensorofdirichlet}, since the $(\sV_i)_{W_i}$ are nuclear Fr\'{e}chet spaces and the completed bornological tensor product coincides with the completed projective tensor product of such spaces, by Corollary 7.1.2 of~\cite{CG1}.

For the second isomorphism, the same argument as for $\tpbc$ gives that 
\[
(\sV_1)_{W_1,\cK_1}\,\hotimes_\beta\, \cdots \,\hotimes_\beta\, (\sV_k)_{W_k,\cK_k}\cong \tpbccs;
\]
since the completed bornological tensor product commutes with colimits, the isomorphism 
\[
(\sV_1)_{W_1,c}\,\hotimes_\beta\, \cdots \,\hotimes_\beta\, (\sV_k)_{W_k,c}\cong \tpbccslf
\]
follows.
\end{proof}

\bibliography{references}

\begin{thebibliography}{MMST20}

\bibitem[AF19]{AyaFraPK}
David Ayala and John Francis.
\newblock Poincar{\'e}/{K}oszul {D}uality.
\newblock {\em Comm. Math. Phys.}, 365(3):847--933, 2019.

\bibitem[AFT17]{AyaFraTan}
David Ayala, John Francis, and Hiro~Lee Tanaka.
\newblock Factorization homology of stratified spaces.
\newblock {\em Sel. Math. New Ser.}, 23(1):293--362, 2017.

\bibitem[BBSS17]{BecBenSchSza}
Christian Becker, Marco Benini, Alexander Schenkel, and Richard~J. Szabo.
\newblock Abelian duality on globally hyperbolic spacetimes.
\newblock {\em Commun. Math. Phys.}, 349(1):361--392, 2017.

\bibitem[BD04]{BD}
Alexander Beilinson and Vladimir Drinfeld.
\newblock {\em Chiral algebras}, volume~51 of {\em American Mathematical
  Society Colloquium Publications}.
\newblock American Mathematical Society, Providence, RI, 2004.

\bibitem[BY]{ButsonYoo}
Dylan Butson and Philsang Yoo.
\newblock Degenerate classical field theories and boundary theories.
\newblock Available at \url{https://arxiv.org/abs/1611.00311}.

\bibitem[CF01]{MR1854134}
Alberto~S. Cattaneo and Giovanni Felder.
\newblock On the {AKSZ} formulation of the {P}oisson sigma model.
\newblock {\em Lett. Math. Phys.}, 56(2):163--179, 2001.
\newblock EuroConf\'{e}rence Mosh\'{e} Flato 2000, Part II (Dijon).

\bibitem[CFFR11]{calaque_felder_ferrario_rossi_2011}
Damien Calaque, Giovanni Felder, Andrea Ferrario, and Carlo~A. Rossi.
\newblock Bimodules and branes in deformation quantization.
\newblock {\em Compositio Mathematica}, 147(1):105--160, 2011.

\bibitem[CG72]{ClemensGriffiths}
C.~Herbert Clemens and Phillip~A. Griffiths.
\newblock The intermediate {J}acobian of the cubic threefold.
\newblock {\em Ann. of Math. (2)}, 95:281--356, 1972.

\bibitem[CG17]{CG1}
Kevin Costello and Owen Gwilliam.
\newblock {\em Factorization algebras in quantum field theory. {V}ol. 1},
  volume~31 of {\em New Mathematical Monographs}.
\newblock Cambridge University Press, Cambridge, 2017.

\bibitem[CILW]{confspacesmfldsbdy}
Ricardo Campos, Najib Idrissi, Pascal Lambrechts, and Thomas Willwacher.
\newblock Configuration spaces of manifolds with boundary.
\newblock Available at \url{https://arxiv.org/abs/1802.00716}.

\bibitem[CLa]{CLbcov1}
Kevin Costello and Si~Li.
\newblock {Q}uantization of open-closed {B}{C}{O}{V} theory, {I}.
\newblock Available at \url{https://arxiv.org/abs/1505.06703}.

\bibitem[CLb]{CostelloLiSUGRA}
Kevin Costello and Si~Li.
\newblock Twisted supergravity and its quantization.
\newblock Available at \url{https://arxiv.org/abs/1606.00365}.

\bibitem[CMR14]{CMR1}
Alberto~S. Cattaneo, Pavel Mnev, and Nicolai Reshetikhin.
\newblock Classical {BV} theories on manifolds with boundary.
\newblock {\em Comm. Math. Phys.}, 332(2):535--603, 2014.

\bibitem[CMR18]{CMR2}
Alberto~S. Cattaneo, Pavel Mnev, and Nicolai Reshetikhin.
\newblock Perturbative quantum gauge theories on manifolds with boundary.
\newblock {\em Comm. Math. Phys.}, 357(2):631--730, 2018.

\bibitem[CNT02]{Cederwall}
Martin Cederwall, Bengt E.~W. Nilsson, and Dimitrios Tsimpis.
\newblock Spinorial cohomology and maximally supersymmetric theories.
\newblock {\em J. High Energy Phys.}, 2:No. 9, 19, 2002.

\bibitem[Cos11]{CosBook}
Kevin Costello.
\newblock {\em Renormalization and effective field theory}, volume 170 of {\em
  Mathematical Surveys and Monographs}.
\newblock American Mathematical Society, Providence, RI, 2011.

\bibitem[Cos13]{CostelloHolomorphic}
Kevin Costello.
\newblock Notes on supersymmetric and holomorphic field theories in dimensions
  2 and 4.
\newblock {\em Pure Appl. Math. Q.}, 9(1):73--165, 2013.

\bibitem[CP21]{CP}
Kevin Costello and Natalie~M. Paquette.
\newblock Twisted supergravity and {K}oszul duality: a case study in {$\rm
  AdS_3$}.
\newblock {\em Comm. Math. Phys.}, 384(1):279--339, 2021.

\bibitem[{Cra}]{crainic}
M.~{Crainic}.
\newblock {On the perturbation lemma, and deformations}.
\newblock Available at \url{https://arxiv.org/abs/math/0403266}.

\bibitem[EMSS89]{EMSS}
Shmuel Elitzur, Gregory Moore, Adam Schwimmer, and Nathan Seiberg.
\newblock Remarks on the canonical quantization of the
  {C}hern-{S}imons-{W}itten theory.
\newblock {\em Nuclear Phys. B}, 326(1):108--134, 1989.

\bibitem[ESW21]{ESW}
Richard Eager, Ingmar Saberi, and Johannes Walcher.
\newblock Nilpotence varieties.
\newblock {\em Ann. Henri Poincar\'{e}}, 22(4):1319--1376, 2021.

\bibitem[FFFS02]{FFS}
Giovanni Felder, J\"{u}rg Fr\"{o}hlich, J\"{u}rgen Fuchs, and Christoph
  Schweigert.
\newblock Correlation functions and boundary conditions in rational conformal
  field theory and three-dimensional topology.
\newblock {\em Compositio Math.}, 131(2):189--237, 2002.

\bibitem[FM20]{FMabeljacobi}
Domenico Fiorenza and Marco Manetti.
\newblock Formal {A}bel-{J}acobi maps.
\newblock {\em Int. Math. Res. Not. IMRN}, (4):1035--1090, 2020.

\bibitem[FMS07]{FMS}
Daniel~S. Freed, Gregory~W. Moore, and Graeme Segal.
\newblock Heisenberg groups and noncommutative fluxes.
\newblock {\em Ann. Physics}, 322(1):236--285, 2007.

\bibitem[Fre00]{FreedDQ}
Daniel~S. Freed.
\newblock Dirac charge quantization and generalized differential cohomology.
\newblock In {\em Surveys in differential geometry}, volume~7 of {\em Surv.
  Differ. Geom.}, pages 129--194. Int. Press, Somerville, MA, 2000.

\bibitem[Gri68a]{Griffiths1}
Phillip~A. Griffiths.
\newblock Periods of integrals on algebraic manifolds. {I}. {C}onstruction and
  properties of the modular varieties.
\newblock {\em Amer. J. Math.}, 90:568--626, 1968.

\bibitem[Gri68b]{Griffiths2}
Phillip~A. Griffiths.
\newblock Periods of integrals on algebraic manifolds. {II}. {L}ocal study of
  the period mapping.
\newblock {\em Amer. J. Math.}, 90:805--865, 1968.

\bibitem[Gwi12]{othesis}
Owen Gwilliam.
\newblock {\em Factorization {A}lgebras and {F}ree {F}ield {T}heories}.
\newblock ProQuest LLC, Ann Arbor, MI, 2012.
\newblock Thesis (Ph.D.)--Northwestern University.

\bibitem[HS05]{HopSing}
M.~J. Hopkins and I.~M. Singer.
\newblock Quadratic functions in geometry, topology, and {M}-theory.
\newblock {\em J. Differential Geom.}, 70(3):329--452, 2005.

\bibitem[Kon03]{KontPSM}
Maxim Kontsevich.
\newblock Deformation quantization of {P}oisson manifolds.
\newblock {\em Lett. Math. Phys.}, 66(3):157--216, 2003.

\bibitem[KS11]{KSonCS}
Anton Kapustin and Natalia Saulina.
\newblock Topological boundary conditions in abelian {C}hern-{S}imons theory.
\newblock {\em Nuclear Phys. B}, 845(3):393--435, 2011.

\bibitem[Lura]{higheralgebra}
Jacob Lurie.
\newblock Higher algebra.
\newblock Available at \url{http://www.math.ias.edu/~lurie/papers/HA.pdf}.

\bibitem[Lurb]{LurieLecture}
Jacob Lurie.
\newblock Koszul {D}uality ({L}ecture 23).
\newblock Available at
  \url{https://www.math.ias.edu/~lurie/282ynotes/LectureXXIII-Koszul.pdf}.

\bibitem[Mat15]{Matsuoka2015}
Takuo Matsuoka.
\newblock Koszul duality for locally constant factorization algebras.
\newblock {\em Serdica Mathematical Journal}, 41(4):369--414, 2015.

\bibitem[MMST20]{edgemodesyangmills}
Philippe Mathieu, Laura Murray, Alexander Schenkel, and Nicholas~J. Teh.
\newblock Homological perspective on edge modes in linear {Y}ang-{M}ills and
  {C}hern-{S}imons theory.
\newblock {\em Lett. Math. Phys.}, 110(7):1559--1584, 2020.

\bibitem[MSW20]{MSW}
Pavel Mnev, Michele Schiavina, and Konstantin Wernli.
\newblock Towards holography in the {BV}-{BFV} setting.
\newblock {\em Ann. Henri Poincar\'{e}}, 21(3):993--1044, 2020.

\bibitem[Rab21]{AR2}
Eugene Rabinovich.
\newblock {\em Factorization Algebras for Bulk-Boundary Systems}.
\newblock PhD thesis, University of California, Berkeley, 2021.

\bibitem[Sch98]{Schwarz}
John~H. Schwarz.
\newblock Remarks on the {M{$5$}}-brane.
\newblock {\em Nuclear Phys. B Proc. Suppl.}, 68:279--284, 1998.
\newblock Strings '97 (Amsterdam, 1997).

\bibitem[Sho10]{Shoikhet}
Boris Shoikhet.
\newblock Koszul duality in deformation quantization and {T}amarkin's approach
  to {K}ontsevich formality.
\newblock {\em Adv. Math.}, 224:731--771, 2010.

\bibitem[SW]{SW3}
Ingmar Saberi and Brian~R. Williams.
\newblock Constraints in the {B}{V} formalism: six-dimensional supersymmetry
  and its twists.
\newblock https://arxiv.org/abs/2009.07116.

\bibitem[Wit89]{WittenJones}
Edward Witten.
\newblock Quantum field theory and the {J}ones polynomial.
\newblock {\em Comm. Math. Phys.}, 121(3):351--399, 1989.

\bibitem[Wit92]{WittenFactorization}
Edward Witten.
\newblock On holomorphic factorization of {WZW} and coset models.
\newblock {\em Comm. Math. Phys.}, 144(1):189--212, 1992.

\bibitem[Wit97]{WittenM5effective}
Edward Witten.
\newblock Five-brane effective action in {$M$}-theory.
\newblock {\em J. Geom. Phys.}, 22(2):103--133, 1997.

\end{thebibliography}
\bibliographystyle{alpha}
\end{document}